\documentclass{amsart}

\usepackage{amssymb}
\usepackage{amsmath}
\usepackage{amsthm}
\usepackage{amscd}
\usepackage{mathrsfs}
\usepackage{graphicx}
\usepackage{fullpage}
\graphicspath{{Figures/}} 

\usepackage{cancel} 
\usepackage{enumerate} 
\usepackage{hyperref} 
\usepackage[usenames,dvipsnames]{color} 
\usepackage{polynom}
\usepackage{subfig} 

\usepackage{tikz}
\usetikzlibrary{arrows}
\usetikzlibrary{calc}

\definecolor{pgray}{RGB}{75,92,113}
\definecolor{porange}{RGB}{230,148,20}
\definecolor{pblue}{RGB}{104,140,152}

\newcommand{\colorlist}[1]{%
\ifcase#1 
Blue%
\or
ForestGreen%
\or
Fuchsia%
\or
Turquoise%
\or
BurntOrange%
\or
WildStrawberry%
\or
LimeGreen%
\else
Orchid%
\fi
}

\hypersetup{colorlinks=true, linkcolor=pblue}

\theoremstyle{plain} 
\newtheorem{theorem}{Theorem}[section]
\newtheorem{proposition}[theorem]{Proposition}
\newtheorem{lemma}[theorem]{Lemma}

\theoremstyle{definition} 
\newtheorem{definition}[theorem]{Definition}
\newtheorem{remark}[theorem]{Remark}

\newtheorem{example}[theorem]{Example}
 
\newcounter{suggestcount}
\newcounter{commentlabel}
\newlength{\poincarecommentlift}
\makeatletter
\DeclareRobustCommand{\COMMENT}[1]{\@bsphack%
	\stepcounter{commentlabel}%
	\vbox to0pt{%
		\setlength{\fboxrule}{0.75pt}%
		\setlength{\fboxsep}{0.75pt}%
		\setlength{\poincarecommentlift}{1ex}%
		\addtolength{\poincarecommentlift}{\fboxrule}%
		\addtolength{\poincarecommentlift}{\fboxsep}%
		\vss\color{red}%
		\rlap{\rlap{\vrule height\poincarecommentlift width\fboxrule}\raise \poincarecommentlift%
		\hbox{\fcolorbox{red}{yellow}{%
\normalfont\footnotesize\ttfamily\bfseries\thecommentlabel}}}}%
	\marginpar{\noindent\raggedright%
	\textbf{\color{red}\thecommentlabel}:\thinspace\footnotesize#1}%
}
\title{  Brill-Noether loci.}
\author{montserrat Teixidor i Bigas}
\address{Mathematics Department, Tufts University, 177 College Avenue, Medford MA 02155, USA}
\email{mteixido@tufts.edu}

\let\oldtocsection=\tocsection
\let\oldtocsubsection=\tocsubsection
\renewcommand{\tocsection}[2]{\hspace{0em}\oldtocsection{#1}{#2}}
\renewcommand{\tocsubsection}[2]{\hspace{1.1em}\oldtocsubsection{#1}{#2}}

\begin{document} 

\maketitle


\setlength{\parindent}{0cm} 
\setlength{\parskip}{\baselineskip} 
\setlength{\abovedisplayskip}{0.7\baselineskip} 
\setlength{\belowdisplayskip}{0.7\baselineskip}

\begin{abstract} Brill-Noether loci ${\mathcal M}^r_{g,d}$  are those subsets of the moduli space ${\mathcal M}_g$ determined by the existence of a linear series of degree $d$ and dimension $r$.
By looking at non-singular curves in a neighborhood of a special chain of elliptic curves,  we provide a new proof of the non-emptiness of the Brill-Noether loci 
when the expected codimension  satisfies  $-g+r+1\le  \rho(g,r,d)\le  0$
  and prove  that for a generic point of a component of this locus, the Petri map is onto.
As an application, we show that Brill-Noether loci of the same codimension are distinct when the codimension is not too large, substantially generalizing the known result in codimensions 1 and 2.
We also provide a new technique for checking that Brill-Noether loci are not included in each other.

Mathematics Subject Classification 14H10, 14H45
\end{abstract}

Let $C$ be a curve of genus $g$. The Brill-Noether locus $G^r_d(C)$ parameterizes linear series of degree $d$ and dimension $r$ on $C$.
When the Brill-Noether number $\rho(g,r,d)\ge 0$, the Brill-Noether locus is non-empty on every curve and irreducible if the number is strictly positive.
On the other hand, when the Brill-Noether number is negative, the locus is empty for $C$ generic but may be non-empty on special curves.
One can then define ${\mathcal M}^r_{g,d}$ as the locus  of curves  that posses a $g^r_d$.
The loci ${\mathcal M}^r_{g,d}$  have played an important role in the study of the geometry of ${\mathcal M}_g$ having allowed for example to show that ${\mathcal M}_g$ is of general type if $g$ is sufficiently large.
Other interesting loci are defined by the rank of natural maps among spaces of sections of line bundles.
One can look for instance at the maximal rank loci determined when the hypersurfaces of the ambient space do not cut a linear series of the expected dimension or the loci of curves where the Petri map fails to be injective.
When $\rho<0$, by dimensionality reasons, it is impossible for the Petri map to be injective but could still be of maximal rank if it were onto.

A first question to ask about the ${\mathcal M}^r_{g,d}$ is whether they are non-empty.
In \cite{Prho<0}, Pflueger showed that ${\mathcal M}^r_{g,d}$ is non-empty and has a component of the expected dimension $-\rho$ when $-\rho\le g-3$ .
Here, we  give a simpler proof for this result for $-\rho\le g-(r+1)$.
We also show that on the generic point of that component of ${\mathcal M}^r_{g,d}$ the Petri map  is onto, fitting with the philosophy of expected maximal rank (see Theorem \ref{ThExComp}).
We will use this result in a forthcoming work to study the loci of curves where the Petri map is not injective.
These loci were expected to be always of codimension one  in the moduli space of curves (see \cite{BS}), a conjecture disproved by Lelli-Chiesa \cite{MarGiesP}.

Another natural question to ask about ${\mathcal M}^r_{g,d}$  is whether the different values of $r, d$ produce distinct loci.
By Serre duality,  ${\mathcal M}^r_{g,d}= {\mathcal M}^{g-d+r-1}_{g,2g-2-d}$, therefore, it suffices to study the case $r+1\le g-d+r$ or equivalently  $d\le g-1$. 
It was proved in \cite{CKK}, \cite{CK}, that when the  codimension is either  1 or 2 and two of the loci ${\mathcal M}^{r_1}_{g,d_1}, {\mathcal M}^{r_2}_{g,d_2}$,
 then the values of the $d_i,r_i$ agree or are Serre dual ($d_2=2g-2-d_1, r_2=g-d_1+r_1$).
Here we extend this result to higher codimension (see Theorem \ref{distBN}).

 Another interesting  question are the possible inclusions of different ${\mathcal M}^r_{g,d}$ in each other.
From a linear series, by adding a fixed point one can always create another linear series of higher degree and the same dimension. 
Similarly, by removing one point, one can create a linear series of both degree and dimension one less. 
Therefore, there are natural  inclusions ${\mathcal M}^r_{g,d}\subseteq {\mathcal M}^r_{g,d+1}$,  ${\mathcal M}^r_{g,d}\subseteq {\mathcal M}^{r-1}_{g,d-1}$. 
In particular, if one is interested on maximal Brill-Noether loci  ${\mathcal M}^r_{g,d}$, it suffices to consider those for which $\rho(g,r,d)< 0$ and $\rho(g,r-1,d-1)\ge 0 , \rho(g,r,d+1)\ge 0$.
 Auel and Haburcak conjectured   that when these conditions are satisfied and $g\ne 7, 8,9$, the loci  ${\mathcal M}^r_{g,d}$ are maximal, that is, not contained on any other  ${\mathcal M}^{r'}_{g,d'}$.
 They prove the conjecture for genus up to 19 and 22, 23 in \cite{AH}.
 In a recent paper \cite{AHK}, the conjecture has been proved in all cases.
 Although we will not pursue this question here, our methods can be used to deal with particular  cases of inclusions of Brill-Noether loci, not just for those expected to be maximal(see Example \ref{exinclloci}).
 Our degeneration technique is very different from the one employed  in \cite{AH}, \cite{AHK} who instead rely on an analysis of line bundles on K3 surfaces.
 
 The author wants to thank Richard Haburcak and Nathan Pflueger for useful comments and suggestions.

\section{Limit Linear Series on a chain of elliptic curves.}

We will use chains of elliptic curves: 
given elliptic curves $E_1, \dots E_g$ with marked points $P_i,Q_i\in E_i$, glue $Q_i$ to $P_{i+1}, i=1,\dots, g-1$ to form a nodal curve $X$ of arithmetic genus $g$.
When the curves and pairs of points on each curve are generic, we will say that we have a generic chain.
In this paper, we will often require that $P_i-Q_i$ be a torsion point of a certain order.

Recall that a  limit linear series of degree $d$ and dimension $r$ on  $X$  consists of the data of a limit linear series of degree $d$ and dimension $r$ on each component  
so that the orders of vanishing of the sections $u^i_0< \dots< u^i_r,\ v^i_0> \dots> v^i_r$  at $P_i, Q_i$ respectively satisfy $u^{i+1}_k+v^i_k\ge d, k=0,\dots, r,\ i=1,\dots, g-1$.
The series is refined if we have equality in these inequalities.
Given a family of curves with central fiber $X$ and a space of linear series of degree $d$ and dimension $r$, 
one obtains as the limit on $X$  (after perhaps adding to  $X$  a few chains of rational curves) a refined limit linear series.

In this section, we present a way to parameterize limit linear series on a chain of elliptic curves that is not necessarily generic.
Representation of Brill-Noether loci on generic chains of elliptic curves by using Young Tableaux was introduced by Edidin (\cite{E} ) in the case where the Brill-Noether number is 0. 
It was much later used, still in the case of $\rho =0$ in the tropical context in in \cite{CDPR}. 
This was generalized to any $\rho $ for $W^r_d$ in \cite{LT} and then Young Tableaux represent whole components of the Brill-Noether locus instead of individual limit linear series.
In the case of special chains, the geometry has been exploited in \cite{JR} in the tropical context.
In \cite{Ramif}, we describe completely the space of limit linear series $G^r_d$ on a chain of generic elliptic curves with ramification at two fixed points.
A similar result is true when the chains are not generic with only a few changes that we describe below.
We do not need the case of added ramification at two points, so we will  not include it, although  ramification at two points does not make the proof any harder.
Recall that in the case of curves of genus 1, Riemann-Roch's Theorem gives a complete answer to the existence of linear series:

\begin{lemma} \label{lemsecLelcurv}Let $C$ be an elliptic curve, $P,Q\in C$, $L$  a line bundle of degree $d$ on $C$, $V\subseteq H^0(C,L)$ an $(r+1)$-dimensional  space of  sections
and $u_0< \dots< u_r,\ v_0> \dots> v_r$  the distinct orders of vanishing of the sections of $V$ at at $P, Q$ respectively.
Then  $u_k+v_k\le d, k=0,\dots, r$.
If equality holds for a particular $k$, then  $L={\mathcal O}(kP+(d-k)Q)$.
Therefore, if $L$ is generic,  $u_k+v_k\le d-1, k=0,\dots, r$ while  if $u_{k_1}+v_{k_1}= d, u_{k_2}+v_{k_2}= d$,   then $P-Q$ is a torsion element of order $u_{k_2}-u_{k_1}$ in the group structure of $C$.
\end{lemma}
\begin{proof} From the definition of $v_k$, the subspace of $V$ of sections with order of vanishing at $Q$ at least $v_k$ has dimension at least $k+1$.
Similarly, the space of sections vanishing with order at least $u_k$ at $P$ has dimension at least $r-k+1$.
Two subspaces of $V$ of dimensions at least $k+1$ and $r-k+1$ necessarily intersect.
Therefore, there is a section in $V$ that vanishes with order at least $u_k$ at $P$ and at least $v_k$ at $Q$.
As the degree of $L$ is $d$, $u_k+v_k\le d$ and equality implies that $L={\mathcal O}(u_kP+v_kQ)$.
\end{proof}

\begin{lemma} \label{lemsecLelcurv2} Let $C$ be an elliptic curve, $P,Q\in C$ such that $P-Q$ is a torsion element of order $l$. 
Let  $0\le u_0< \dots< u_r\le d,\ d\ge v_0> \dots> v_r\ge 0$  be distinct integers such that 
\[ u_{k_1}+v_{k_1}=u_{k_2}+v_{k_2}=d, u_j+v_j=d-1, j\not=k_1, k_2\]
 Then there exists a line bundle  $L$  on $C$ and an $(r+1)$-dimensional  space of  sections of $L$ with orders of vanishing at $P$ (resp $ Q$ ) $u_i$ (resp $v_i$)
 if and only if $l$ divides  $u_{k_2}-u_{k_1}$ and $l$  does not divide  $u_{j}-u_{k_1}$ or $u_j+1-u_{k_1}, j\not=k_2$.
  \end{lemma}
\begin{proof} The only possible $L$ that could satisfy the conditions is $L={\mathcal O}(u_{k_1}P+v_{k_1}Q)={\mathcal O}(u_{k_2}P+v_{k_2}Q)$ where the second equality comes from the assumption that 
 $P-Q$ is a torsion element of order $l$ and that  $l$ divides  $u_{k_2}-u_{k_1}$.
 From the assumption that $u_j+v_j=d-1, j\not=k_1, k_2$, $ H^0(L(-u_jP-v_jQ))=1$.
 So, there exists one section $s_j$ of $L$ vanishing to order at least $u_j$ at $P$ and at least $v_j$ at $Q$.
 The orders of vanishing  of $s_j$ at $P,Q$ are precisely $u_j, v_j$ if and only if $L\not={\mathcal O}(u_jP+(v_j+1)Q),  L\not={\mathcal O}((u_j+1)P+v_jjQ)$. 
 This is equivalent to the condition that  $l$  does not divide  $u_{j}-u_{k_1}$ or $u_j+1-u_{k_1}$.
 
 As the $s_j$ constructed in this way have different vanishing orders at the points $P, Q$, they are linearly independent.
 Therefore, they span an $r+1$-dimensional space of sections of $L$.
 \end{proof}

We present first the simpler description of the components of the space of limit linear series on a non-generic chain of elliptic curves.  This is the case we will mostly use. 

 \begin{definition} \label{defadfrec} Given a collection of boxes arranged forming a  rectangle, an {\bf admissible filling }  assigns to each box  precisely one number in $1,\dots,g$
  so that  numbers in rows read from left to right and columns read from top to bottom  appear in strictly increasing order. Some of the numbers $1,\dots,g$  may be used  multiple times.
  We will indicate with $(a,b)$ the coordinates of a spot (column and row respectively).
 \end{definition}
 
 A similar concept called ``displacement tableaux'' was used in \cite{PBN} for chains of elliptic curves and in \cite{Pdisp} in the context of tropical geometry. 
 \bigskip
\begin{proposition} \label{prop:compesp} Choose $g, r, d$  positive integers.
Let $X$ be a  chain of $g$ elliptic curves generic except for   components $C_{i_1},\dots,C_{i_e}$  where the $l_i$ are the smallest positive indices such that 
\[ l_1(P_{i_1}-Q_{i_1})=0,\dots, l_e(P_{i_e}-Q_{i_e})=0 .\]
There is a one-to-one correspondence between the two following sets of data
\begin{enumerate}[(a)]
\item Components of the Brill-Noether locus  of refined limit linear series of degree $d$ and dimension $r$ on $X$ 
such that vanishing at $P_1, Q_g$ is $(0,\dots, r), (r,\dots, 0)$ and vanishing at  $P_i,Q_i$ satisfy $u_j^i+v_j^i\ge d-1$.
\item  Admissible fillings of an $(r+1)\times (g-d+r)$ rectangle so that
\begin{itemize}
\item Only the numbers $i_1,\dots, i_e$  may appear more than once 
 \item When $i_k$ appears on spots  $(a_1,b_1) $ and   $(a_2,b_2) $ then $l_k$ divides the grid distance $a_2-a_1+b_2-b_1$.
 \item No index $i$  on spot  $(a_3,b_3) $  such that $l_k$ divides $a_3-a_1+b_3-b_1$ could be replaced by $i_k$ and leave a  filling that is still admissible.
 \end{itemize}
 \end{enumerate}
\end{proposition}
\begin{remark} Linear series in which  $u_j^1+v_j^i\le d-2$ are in a sense  wasteful and less likely to be possible when the Brill-Noether number is negative.
They give rise to positive dimensional subsets  of the Brill-Noether locus in which the line bundle is fixed and only the space of sections vary (see \cite{CLPT})
We will consider this possibility  in  \ref{teor:compgen}  and then see in \ref{teor:sllimplicasllesp}  that when the general case is possible for a chain, so is  the special case of  $u_j^1+v_j^i\ge d-1$. 
An example for both types of filling and correspondences is shown in Example\ref{exfillls}.
Some readers may prefer to look at the example before the  proofs that follow.
\end{remark}
\begin{proof} Fix an admissible filling of the  $(r+1)\times (g-d+r)$ rectangle as listed.
We construct a limit linear series associated to it:

 First define vanishing orders at the nodes inductively  as follows
\[\text {For } j=0,\dots, r,\  u_j^1=j, \  u_j^i=\begin{cases}u_j^{i-1}  & i-1 \text{ is on column } j  \\u_j^{i-1}+1 &  \text{ otherwise}\end{cases} \ i=2,\dots, g. \]
\[ \text {For } j=0,\dots, r,\  v_j^i=d-u_j^{i+1}, \ \  i=1,\dots g-1, \ \ v_j^g=d-j.\]
In particular, $u_j^i=j+i-1-|\{ k<i|\  k\text{ is in column }j \}|$. 

One of the conditions on the fillings of the rectangle  is that rows need to be in strictly increasing order. This  guarantees that $u_j^i<u_{j+1}^i$.
Moreover,  
\[ (*)\ \  u_j^i+1=u_{j+1}^i \Leftrightarrow |\{ k<i| k\text{ is in column }j \}|=|\{ k<i| k\text{ is in column }j +1\}|.\]
When (*) holds, as rows are strictly increasing, $i$ cannot be placed in column $j+1$. This ensures that in all cases,  $v_j^i>v_{j+1}^i$.

If $i\in \{ 1,2,\dots,g\}$ does not appear  in the filling of the Tableau, take $L_i$ a generic line bundle of degree $d$.
In that case, for every $j$,  $u_j^{i+1}=u_j^i+1$ and therefore, $v_j^i=d-u_j^{i+1}=d-u_j^i-1$. 
There is then a well determined $(r+1)$-dimensional space of sections of $L_i$ with the giving vanishing at $P_i, Q_i$.

If $i\in \{ 1,2,\dots,g\}$ appears in at least one spot in the filling of the Tableau in spot $(j, y)$ then , $u_j^{i+1}=u_j^i$. 
Therefore, $v_j^i=d-u_j^{i+1}=d-u_j^i$. 
Define $L_i={\mathcal O}(u_j^iP_i+v_j^iQ_i)$ which is a line bundle of degree $d$.

If $i$ appears in the Tableau at least twice,  on  spots  $(a_1,b_1) $ and   $(a_2,b_2) $
\[ u_{a_2}^i-u_{a_1}^i=a_2-a_1+|\{ k<i|\  k\text{ is in column }a_2 \}|-|\{ k<i| \ k\text{ is in column }a_1\}|=a_2-a_1+b_2-b_1\]
The conditions $ l_i(P_i-Q_i)=0$, $l_i$ divides $b_2-b_1+a_2-a_1$ guarantee that ${\mathcal O}(u_{a_2}^iP_i+v_{a_2}^iQ_i)={\mathcal O}(u_{a_1}^iP_i+v_{a_1}^iQ_i)$.

Therefore, $L_i$ is well defined in all instances.
As before, take the space of sections of $L_i$ with the given vanishing at $P_i, Q_i$.

Conversely, assume given the data of a limit linear series on $X$ satisfying  $u_j^i+v_j^i\ge d-1$ and  whose vanishing at $P_1, Q_g$ is $(0,\dots, r), (r,\dots, 0)$.
We fill the rectangle  by adding successively the indices $1,\dots, g$ in the boxes on the first empty spot of column $j$ if $u_j^i+v_j^i=d$.
In particular, if this happens, as the line bundles have degree $d$, $L_i={\mathcal O}( u_j^iP_i+v_j^iQ_i)$.
Then, the  condition that $P_i, Q_i$ are generic on $C_i$ except on $C_{i_1},\dots,C_{i_e}$  where  $l_{i_j}(P_{i_j}-Q_{i_j})=0 $ 
means that only the indices $i_1,\dots, i_e$ may need to be placed in  more than one spot.
Moreover, if $i_k$ needs to be placed at $(a_1,b_1) $ and   $(a_2,b_2) $, then $u_{a_1}^{i_k}+v_{a_1}^{i_k}=d, u_{a_2}^{i_k}+v_{a_2}^{i_k}=d$.
This requires $L_{i_k}={\mathcal O}( u_{a_1}^{i_k}P_{i_k}+v_{a_1}^{i_k}Q_{i_k})={\mathcal O}( u_{a_2}^{i_k}P_{i_k}+v_{a_2}^{i_k}Q_{i_k})$.
From the genericity of the curve except on the components $C_{i_k}$, $l_k$ divides $u_{a_1}^{i_k}-u_{a_2}^{i_k}$.
As the series is refined,  $u_j^{t+1}=d-v_j^t$.
From  the vanishing  $(0,\dots, r)$ at $P_1$  and the condition $v_j^t=d-u_j^t-1$ unless the index $t$ is placed in column $j$, in which case  $v_j^t=d-u_j^t$,
we get $ u_{a_1}^{i_k}= a_1+i_k-b_1, v_{a_1}^{i_k}=d-a_1-i+b_1$.
Therefore, $u^{i_k}_{a_2}-u^{i_k}_{a_1}=a_1+b_1-a_2-b_2 $ and the repeated indices are  placed according to the required distances.

It remains to check that the whole rectangle is filled.
By assumption, $v_j^g=r-j$. 
By construction, $v_j^g=d-u_j^1-g+|\{ k |\ k\text{ is placed in column } j \}|=d-j-g+|\{ k |\ k\text{ is placed in column } j \}|$.
Therefore, $|\{ k |\ k\text{ is placed in column } j \}|=g-d+r$, meaning that all the rows on column $j$ get filled.
\end{proof}

\begin{definition}\label{def:admfil}  Given  integers $g\ge 2$, $i_1,\dots, i_e\in \{ 1,\dots, g\}, l_1,\dots, l_e, 0\le e\le g, l_i\ge 2$ and a rectangle consisting of $\alpha\times \beta $ squares, 
a  filling of the rectangle  assigns to each box  on the vertical strip containing the rectangle (including boxes above and below it) numbers among $ 1,\dots, g$, each with a positive or negative weight of 1. 
The $i$-weight $\mathbf {w^i(a,b)}$ of a box is the sum of the weights of the numbers $i'\le i$ that appear on it
By definition, the 0-weight of a box is 1 for any squares above the given rectangle and 0 for any other box.
We  call the filling  an  {\em admissible filling}  if it satisfies
\begin{enumerate}[(a)]
\item Only the numbers $i_1,\dots, i_e$  appear more than once with weight 1
 and when $i_k$ appears on  spots  $(a_1,b_1) $ and   $(a_2,b_2) $,  then $l_k$ divides $b_2-b_1+a_2-a_1$.
Any number may appear multiple times with weight $-1$.
\item Each box has at most one occurrence of a given number (including both positive and negative weight).
\item For each choice of an $i, 0\le i\le g$, the $i$-weight of a box  is  either 0 or 1.
\item  The $i$-weight of any given box is greater than or equal to the weight of a box  to its right.
\item   The $i$-weight of any given box is greater than or equal to the weight of a box   below.
\item  The $g$-weight of a box is 1 for the boxes lying inside or above the rectangle  and 0 for those lying below that lower side of the rectangle.
\end{enumerate}
\end{definition}
\begin{remark}
 We are assuming that the filling can be placed  in rows above and below the rectangle. 
 The weights of any number placed outside the rectangle cancel each other.
\end{remark}

\begin{proposition} \label{teor:compgen} Choose $g, r, d$.
Let $X$ be a  chain of $g$ elliptic curves generic except for   components $C_{i_1},\dots,C_{i_e}$  where 
\[ l_1(P_{i_1}-Q_{i_1})=0,\dots, l_e(P_{i_e}-Q_{i_e})=0 .\]
There is a one-to-one correspondence between the components of the Brill-Noether locus  of limit linear series of degree $d$ and dimension $r$ on $X$ 
 and admissible  fillings of  an $(r+1)\times (g-d+r)$ rectangle as described in \ref{def:admfil}.
\end{proposition}
\begin{proof} The proof is  analogous   to the proof of Theorem 3.4 in \cite{Ramif} except that now we do not assume the pairs of points $P_i, Q_i$ on each elliptic component to be generic. 
This allows us to place the index of  that component in multiple spots as was the case in Proposition \ref{prop:compesp} above.
\end{proof}

\begin{example}\label{exfillls}
\begin{enumerate}[(a)]
\item We describe the limit linear series corresponding to the positive filling on the left panel of Figure \ref{figrectfil}.
Take a chain of $10$ elliptic curves, generic except for the choice of the gluing points on $C_5$ where $3(P_5-Q_5)=0$.
Here $r=1, d=7$.
On the curves $C_1, C_2, C_7$ take generic line bundles.
Take 
\[ L_3={\mathcal O}(2P_3+5Q_3), L_4={\mathcal O}(2P_4+5Q_4), \ L_5={\mathcal O}(2P_5+5Q_5)={\mathcal O}(5P_5+2Q_5), L_6={\mathcal O}(2P_6+5Q_6),\]
\[ L_8={\mathcal O}(7P_8) , L_9={\mathcal O}(7P_9) , L_{10}={\mathcal O}(7P_{10}) \]
Writing $s^i_{a,b}, t^i_{a,b}$ for a section of $L_i$ vanishing at $P_i, Q_i$ with orders $a, b, a+b=d$ or $a+b=d-1$, we take the space of sections spanned by the following sections
\[ (s^1_{0,6}, t^1_{1,5}), \  (s^2_{1,5}, t^2_{2,4}), \  (s^3_{2,5}, t^2_{3,3}), \  (s^4_{2,5}, t^4_{4,2}), \  (s^5_{2,5}, t^5_{5,2}), \  (s^6_{2,5}, t^6_{5,1}), \  (s^7_{2,4}, t^7_{6,0}), \  (s^8_{3,3}, t^7_{7,0}),
 \  (s^9_{4,2}, t^9_{7,0}) , \  (s^{10}_{5,1}, t^{10}_{7,0})\ \]

\begin{figure}[h!]
\begin{tikzpicture}[scale=.8]

\begin{scope}[xscale=1.5, xshift=4cm]
\foreach \x in {0, 1,2,3,4} {\draw[thick] (0,\x) -- (2, \x); }
\foreach \x in {0, 1,2} {\draw[thick] (\x,0) -- ( \x,4); }
\node at (.2, 3.7){1 };
\node[color=red] at (.8, 3.7){$2^{-}$ };
\node at (.5, 3.2){3 };
\node at (.5, 2.5){4}; 
\node[color=cyan]  at (1.5, 3.5){5};
\node[color=blue] at (.5, 1.5){5}; 
\node[color=cyan] at (1.2, 2.7){6};
\node[color=red] at (1.8, 2.7){$7^{-}$ };
\node at (1.5, 2.2){8}; 
\node[color=blue] at (.5, .5){6}; 
\node at (1.5, 1.5){9}; 
\node at (1.5, .5){10}; 
\end{scope}

\begin{scope}
\foreach \x in {0, 1,2,3,4} {\draw[thick] (0,\x) -- (2, \x); }
\foreach \x in {0, 1,2} {\draw[thick] (\x,0) -- ( \x,4); }
\node at (.5, 3.5){3 };
\node at (.5, 2.5){4}; 
\node[color=cyan] at (1.5, 3.5){5};
\node[color=blue] at (.5, 1.5){5}; 
\node at (1.5, 2.5){8}; 
\node at (.5, .5){6}; 
\node at (1.5, 1.5){9}; 
\node at (1.5, .5){10}; 

\end{scope}

\end{tikzpicture}
\caption{ An arbitrary filling(right) and  corresponding positive filling(left).}
\label{figrectfil}
\end{figure}
\item We describe the limit linear series corresponding to the  filling on the right panel of Figure \ref{figrectfil}.
Take a chain of $10$ elliptic curves, generic except for the choice of the gluing points on $C_5$ where $3(P_5-Q_5)=0$ and on $C_6$ where $3(P_6-Q_6)=0$.
Again, $r=1, d=7$.
On the curves $ C_2, C_7$ take generic line bundles.
Take 
\[ L_1={\mathcal O}(7Q_1), L_3={\mathcal O}(2P_3+5Q_3), L_4={\mathcal O}(2P_4+5Q_4), \ L_5={\mathcal O}(2P_5+5Q_5)={\mathcal O}(5P_5+2Q_5), \]
\[ L_6={\mathcal O}(2P_6+5Q_6)={\mathcal O}(5P_6+2Q_6), L_8={\mathcal O}(7P_8) , L_9={\mathcal O}(7P_9) , L_{10}={\mathcal O}(7P_{10}) \]
Writing $s^i_{a,b}, t^i_{a,b}$ for a section of $L_i$ vanishing at $P_i, Q_i$ with orders $a, b$, we take spaces of sections spanned by the following sections
\[ (s^1_{0,7}, t^1_{1,5}), \  (s^2_{0,5}, t^2_{2,4}), \  (s^3_{2,5}, t^2_{3,3}), \  (s^4_{2,5}, t^4_{4,2}), \  (s^5_{2,5}, t^5_{5,2}), \  (s^6_{2,5}, t^6_{5,2}), \  (s^7_{2,4}, t^7_{5,0}), \  (s^8_{3,3}, t^7_{7,0}),
 \  (s^9_{4,2}, t^9_{7,0}) , \  (s^{10}_{5,1}, t^{10}_{7,0})\ \]
\end{enumerate}
\end{example}

\begin{proposition} \label{teor:sllimplicasllesp} Choose $g, r, d$.
Let $X$ be a  chain of $g$ elliptic curves generic except for   components $C_{i_1},\dots,C_{i_e}$  where 
\[ l_1(P_{i_1}-Q_{i_1})=0,\dots, l_e(P_{i_e}-Q_{i_e})=0 .\]
If there is a limit linear series on $X$, then there is a limit linear series whose vanishing at $P_1, Q_g$ is $(0,\dots, r), (r,\dots, 0)$ and $P_i,Q_i$ satisfy $u_j^i+v_j^i\ge d-1$.
\end{proposition}
\begin{proof} In view  of Propositions \ref{prop:compesp} , \ref{teor:compgen}, it suffices to see that the existence of an admissible filling of   $(r+1)\times (g-d+r)$ rectangle as described in \ref{def:admfil}
implies the existence of that same rectangle with only positive weights.
This can be built by removing from the filling of each square of the rectangle all but the last index that appears with positive weight 1.

The right pane of Figure \ref{figrectfil} shows an arbitrary filling of the rectangle whose associated positive filling is on the left pane.
\end{proof}

\begin{lemma}\label{lemseptor} Assume that  $\alpha\le \beta$ and that $e$ is a positive integer with 
 \[  e\le \frac{\alpha^2+\alpha-4}2 \text{ if }\alpha\le \beta -2;\ \   e\le \frac{\alpha^2-\alpha}2 \text{ if }\alpha=\beta , \beta -1.\]
 Define $k, j$ by
\[ \frac{k(k+1)}2\le  e< \frac{(k+1)(k+2)}2, \  \  j=e- \frac{k(k+1)}2. \]
Consider an admissible filling of an $\alpha\times \beta$ rectangle as in Definition \ref{defadfrec}  such that the only indices that appear twice are  $i_1,\dots i_e$.
If  $i_t$ appears on  spots  $(a^{i_t}_1,b^{i_t}_1) $ and   $(a^{i_t}_2,b^{i_t}_2) $, then, 
\[ \sum_{t=1}^e(|b^{i_t}_2-b^{i_t}_1|+|a^{i_t}_2-a^{i_t}_1|)\le e(\alpha+\beta-2)-2(\frac {k^3-k}3+jk).       \]
There are admissible fillings of the rectangle with the numbers $1,2,\dots, \alpha\beta-e$ for which this inequality is an equality.
\end{lemma}
\begin{proof} We are trying to maximize the grid distance between the pairs of spots where we place repeated indices.

As there are $\beta$ rows and $\alpha$ columns, $1\le b^i_1<b^i_2\le \beta, 1\le a^i_2<a^i_1\le \alpha$.
Therefore 
\[ |b^i_2-b^i_1|\le \beta-1,\  |a^i_2-a^i_1|\le \alpha-1.\]
So an upper  bound for the grid distance between two squares of the rectangle is  $\alpha+\beta-2$.
The highest bound for both pairs of inequalities  occurs  when the two squares are placed in diagonally opposite spots of the rectangle $(1,\beta), (\alpha,1)$.
There are two spots close to each of the corners  (namely  $ (2,\beta), (1,\beta-1), (\alpha-1,1), (\alpha,2)$)
 where one of the indices for the row and column $r,c$ stays at the max(resp min) while the other decreases (resp increases) by one. 
There are three spots where the difference with the max or min is 3, \dots, there are $t$ squares close to each corner giving a distance of $t-1$ to the optimal distance up to $t=\alpha -1$.
For $t\ge \alpha$ (and assuming $\beta\ge \alpha +2$ so that the spots from the top and the bottom do not overlap),
 there are only $\alpha-1$ spots with optimal distance as we cannot place a double index close to the bottom left corner on the last column  or  a double index close to the top right corner on the first  column.
We only consider here $t\le \alpha$ and in the case $t=\alpha$ we only use the $\alpha-2$ spots at every corner that are in columns $2,3,\dots, \alpha-2$.
Note that  
\[1+2+3+\dots+k=\frac{k(k+1)}2, \ [1+2+3+\dots+(\alpha-1)]+(\alpha-2)=\frac{\alpha^2+\alpha-4}2.\]
 By assumption  $   e\le \frac{\alpha^2-\alpha}2 \text{ if }\alpha=\beta , \beta -1, \ \ e\le \frac{\alpha^2+\alpha-4}2 \text{ if }\alpha\le \beta -2$. 
 We defined  $k$ by  $ \frac{k(k+1)}2\le  e< \frac{(k+1)(k+2)}2$.
 Therefore $k\le \alpha-1$ and also by definition of $j$, $j\le k-1$.

\begin{figure}[h!]
\begin{tikzpicture}[scale=.55]

\begin{scope}[]
\fill [color=pink] (1,5)--(5,5)--(5,1)--(4,1)--(4,2)--(3,2)--(3,3)--(2,3)--(2,4)--(1,4)--(1,5);
\fill [color=pink] (0,0)--(0,4)--(1,4)--(1,3)--(2,3)--(2,2)--(3,2)--(3,1)--(4,1)--(4,0)--(0,0);

\foreach \x in {0, 1,2,3,4,5} {\draw[thick] (0,\x) -- (5, \x); }
\foreach \x in {0, 1,2,3, 4, 5} {\draw[thick] (\x,0) -- ( \x,5); }
\end{scope}

\begin{scope}[xshift=5.5cm]
\fill [color=pink] (1,6)--(5,6)--(5,2)--(4,2)--(4,3)--(3,3)--(3,4)--(2,4)--(2,5)--(1,5)--(1,6);
\fill [color=pink] (0,0)--(0,4)--(1,4)--(1,3)--(2,3)--(2,2)--(3,2)--(3,1)--(4,1)--(4,0)--(0,0);

\foreach \x in {0, 1,2,3,4,5,6} {\draw[thick] (0,\x) -- (5, \x); }
\foreach \x in {0, 1,2,3, 4, 5} {\draw[thick] (\x,0) -- ( \x,6); }
\end{scope}

\begin{scope}[xshift=11cm]
\fill [color=pink] (1,7)--(5,7)--(5,3)--(3,3)--(3,4)--(2,4)--(2,5)--(1,5)--(1,7);
\fill [color=pink] (0,0)--(0,4)--(2,4)--(2,3)--(3,3)--(3,2)--(4,2)--(4,0)--(0,0);
\foreach \x in {0, 1,2,3,4,5,6,7} {\draw[thick] (0,\x) -- (5, \x); }
\foreach \x in {0, 1,2,3, 4, 5} {\draw[thick] (\x,0) -- ( \x,7); }
\end{scope}

\begin{scope}[xshift=16.5cm]
\fill [color=pink] (1,8)--(5,8)--(5,4)--(3,4)--(3,5)--(2,5)--(2,6)--(1,6)--(1,8);
\fill [color=pink] (0,0)--(0,4)--(2,4)--(2,3)--(3,3)--(3,2)--(4,2)--(4,1)--(4,0)--(0,0);

\foreach \x in {0, 1,2,3,4,5,6,7,8} {\draw[thick] (0,\x) -- (5, \x); }
\foreach \x in {0, 1,2,3, 4, 5} {\draw[thick] (\x,0) -- ( \x,8); }
\end{scope}

\begin{scope}[xshift=22cm]
\fill [color=pink] (1,9)--(5,9)--(5,5)--(3,5)--(3,6)--(2,6)--(2,7)--(1,7)--(1,9);
\fill [color=pink] (0,0)--(0,4)--(2,4)--(2,3)--(3,3)--(3,2)--(4,2)--(4,1)--(4,0)--(0,0);

\foreach \x in {0, 1,2,3,4,5,6,7,8,9} {\draw[thick] (0,\x) -- (5, \x); }
\foreach \x in {0, 1,2,3, 4, 5} {\draw[thick] (\x,0) -- ( \x,9); }
\end{scope}

\end{tikzpicture}
\caption{ When $\alpha=5$, the shading show the corners that can be matched. For $\beta=5, 6$, we cover up to the two triangles in the corners, for $\beta>6$, in the middle columns we can add one extra square.}
\label{fig:5colanglesmarc}
\end{figure}

To maximize the sum of grid distances, we should use all the optimal spots  up to distance $k$ and $j$ of those at distance $k+1$.
This gives
\[ e(\alpha+\beta-2)-2[1\times 0+2\times 1+\dots +k(k-1)+jk]=e(\alpha+\beta-2)-2[\frac {k^3-k}3+jk].       \]      
Because we are using grid distance and we are taking their sum, it does not matter in which way we pair the squares at the two corners.
Adding the maximum distances, we obtain the upper bound.

To show that the bound is attained, we need to manufacture an admissible filling that fits the spots with the optimal separation.
We will do this in Lemma \ref{lemexfilcorn}.
\end{proof}

\begin{lemma}\label{lemexfilcorn} Assume that  $\alpha\le \beta$ and that $e$ is a positive integer with 
 \[  e\le \frac{\alpha^2+\alpha-4}2 \text{ if }\alpha\le \beta -2;\ \   e\le \frac{\alpha^2-\alpha}2 \text{ if }\alpha=\beta , \beta -1.\]
There exists then  an admissible filling of an $\alpha\times \beta$ rectangle as in Definition \ref{defadfrec}  such that at most $e$ of the indices appear twice 
and these indices are  placed at the maximal distance as described in Lemma \ref{lemseptor} .
\end{lemma}
\begin{proof} From the proof of Lemma \ref{lemseptor}, the double fillings must be as close as possible to the top right and bottom left corners (see pink region on the rectangles in Figure \ref{fig:5colanglesmarc}).


\begin{figure}[h!]
\begin{tikzpicture}[scale=.5]

\begin{scope}[ ]
\foreach \x in {0, 1,2,3,4} {\draw[thick] (0,\x) -- (4, \x); }
\foreach \x in {0, 1,2,3,4} {\draw[thick] (\x,0) -- ( \x,4); }

\node  [color=purple] at (.5, 3.5){1}; 
\node [color=cyan]  at (1.5, 3.5){2}; 
\node  [color=cyan] at (2.5, 3.5){3}; 
\node[color=cyan] at (3.5, 3.5){6}; 

\node [color=blue] at (.5, 2.5){2};
\node  [color=purple] at (1.5, 2.5){4}; 
\node [color=cyan]  at (2.5, 2.5){5}; 
\node [color=cyan]  at (3.5, 2.5){7}; 

\node [color=blue] at (.5, 1.5){3};
\node  [color=blue]  at (1.5, 1.5){6}; 
\node  [color=purple]   at (2.5, 1.5){8}; 
\node [color=cyan]  at (3.5, 1.5){9}; 

\node [color=blue] at (.5, 0.5){5};
\node  [color=blue]  at (1.5, 0.5){7}; 
\node   [color=cyan]  at (2.5, 0.5){9}; 
\node  [color=purple] at (3.5, 0.5){10}; 
\end{scope}

\begin{scope}[xshift=5cm]
\foreach \x in {0, 1,2,3,4, 5} {\draw[thick] (0,\x) -- (5, \x); }
\foreach \x in {0, 1,2,3,4,5} {\draw[thick] (\x,0) -- ( \x,5); }

\node  [color=purple] at (.5, 4.5){1}; 
\node  [color=cyan]  at (1.5, 4.5){2}; 
\node  [color=cyan] at (2.5, 4.5){3}; 
\node[color=cyan] at (3.5, 4.5){6}; 
\node[color=cyan] at (4.5, 4.5){8}; 

\node [color=blue] at (.5, 3.5){2};
\node   [color=purple]   at (1.5, 3.5){4}; 
\node [color=cyan] at (2.5, 3.5){5}; 
\node [color=cyan] at (3.5, 3.5){7}; 
\node[color=cyan] at (4.5, 3.5){11}; 

\node[color=blue] at (.5, 2.5){3}; 
\node [color=blue] at (1.5, 2.5){6}; 
\node    [color=purple]    at (2.5, 2.5){9}; 
\node [color=cyan] at (3.5, 2.5){10}; 
\node[color=cyan] at (4.5, 2.5){12}; 

\node[color=blue] at (.5, 1.5){5}; 
\node[color=blue] at (1.5, 1.5){8}; 
\node   [color=blue] at (2.5, 1.5){11}; 
\node    [color=purple] at (3.5, 1.5){13}; 
\node [color=cyan]at (4.5, 1.5){14}; 

\node[color=blue] at (.5, .5){7}; 
\node[color=blue] at (1.5, .5){10}; 
\node [color=blue] at (2.5, .5){12}; 
\node   [color=blue]  at (3.5, .5){14}; 
\node  [color=purple] at (4.5, .5){15}; 
\end{scope}

\begin{scope}[xshift=11cm]
\foreach \x in {0, 1,2,3,4, 5,6} {\draw[thick] (0,\x) -- (6, \x); }
\foreach \x in {0, 1,2,3,4,5,6} {\draw[thick] (\x,0) -- ( \x,6); }

\node  [color=purple] at (.5, 5.5){1}; 
\node  [color=cyan]at (1.5, 5.5){2}; 
\node [color=cyan] at (2.5, 5.5){3}; 
\node[color=cyan] at (3.5, 5.5){6}; 
\node[color=cyan] at (4.5, 5.5){8}; 
\node[color=cyan] at (5.5, 5.5){12}; 

\node [color=blue] at (.5, 4.5){2}; 
\node  [color=purple] at (1.5, 4.5){4}; 
\node [color=cyan]  at (2.5, 4.5){5}; 
\node[color=cyan] at (3.5, 4.5){7}; 
\node[color=cyan] at (4.5, 4.5){11}; 
\node[color=cyan] at (5.5, 4.5){14}; 

\node [color=blue] at (.5, 3.5){3};
\node   [color=blue]  at (1.5, 3.5){6}; 
\node  [color=purple] at (2.5, 3.5){9}; 
\node [color=cyan]at (3.5, 3.5){10}; 
\node[color=cyan] at (4.5, 3.5){13}; 
\node[color=cyan] at (5.5, 3.5){17}; 

\node[color=blue] at (.5, 2.5){5}; 
\node [color=blue] at (1.5, 2.5){8}; 
\node [color=blue]  at (2.5, 2.5){12}; 
\node  [color=purple] at (3.5, 2.5){15}; 
\node[color=cyan] at (4.5, 2.5){16}; 
\node[color=cyan] at (5.5, 2.5){18}; 

\node[color=blue] at (.5, 1.5){7}; 
\node[color=blue] at (1.5, 1.5){11}; 
\node [color=blue] at (2.5, 1.5){14}; 
\node [color=blue] at (3.5, 1.5){17}; 
\node  [color=purple] at (4.5, 1.5){19}; 
\node[color=cyan] at (5.5, 1.5){20}; 

\node[color=blue] at (.5, .5){10}; 
\node[color=blue] at (1.5, .5){13}; 
\node[color=blue] at (2.5, .5){16}; 
\node  [color=blue] at (3.5, .5){18}; 
\node [color=blue] at (4.5, .5){20}; 
\node  [color=purple] at (5.5, .5){21}; 
\end{scope}

\begin{scope}[xshift=18cm]
\foreach \x in {0, 1,2,3,4, 5,6,7} {\draw[thick] (0,\x) -- (7, \x); }
\foreach \x in {0, 1,2,3,4,5,6,7} {\draw[thick] (\x,0) -- ( \x,7); }

\node  [color=purple] at (.5, 6.5){1}; 
\node  [color=cyan]at (1.5, 6.5){2}; 
\node [color=cyan] at (2.5, 6.5){3}; 
\node[color=cyan] at (3.5, 6.5){6}; 
\node[color=cyan] at (4.5, 6.5){8}; 
\node[color=cyan] at (5.5, 6.5){12}; 
\node[color=cyan] at (6.5, 6.5){15}; 

\node  [color=blue] at (.5, 5.5){2}; 
\node   [color=purple] at (1.5, 5.5){4}; 
\node [color=cyan] at (2.5, 5.5){5}; 
\node[color=cyan] at (3.5, 5.5){7}; 
\node[color=cyan] at (4.5, 5.5){11}; 
\node[color=cyan] at (5.5, 5.5){14}; 
\node[color=cyan] at (6.5, 5.5){19};

\node [color=blue] at (.5, 4.5){3}; 
\node [color=blue]  at (1.5, 4.5){6}; 
\node  [color=purple] at (2.5, 4.5){9}; 
\node[color=cyan] at (3.5, 4.5){10}; 
\node[color=cyan] at (4.5, 4.5){13}; 
\node[color=cyan] at (5.5, 4.5){18}; 
\node[color=cyan] at (6.5, 4.5){21}; 

\node [color=blue] at (.5, 3.5){5};
\node   [color=blue]  at (1.5, 3.5){8}; 
\node  [color=blue] at (2.5, 3.5){12}; 
\node  [color=purple] at (3.5, 3.5){16}; 
\node[color=cyan] at (4.5, 3.5){17}; 
\node[color=cyan] at (5.5, 3.5){20}; 
\node[color=cyan] at (6.5, 3.5){24}; 

\node[color=blue] at (.5, 2.5){7}; 
\node [color=blue] at (1.5, 2.5){11}; 
\node [color=blue]  at (2.5, 2.5){15}; 
\node [color=blue] at (3.5, 2.5){19}; 
\node  [color=purple] at (4.5, 2.5){22}; 
\node[color=cyan] at (5.5, 2.5){23}; 
\node[color=cyan] at (6.5, 2.5){25}; 

\node[color=blue] at (.5, 1.5){10}; 
\node[color=blue] at (1.5, 1.5){14}; 
\node [color=blue] at (2.5, 1.5){18}; 
\node [color=blue] at (3.5, 1.5){21}; 
\node [color=blue] at (4.5, 1.5){24}; 
\node  [color=purple] at (5.5, 1.5){26}; 
\node[color=cyan] at (6.5, 1.5){27}; 

\node[color=blue] at (.5, .5){13}; 
\node[color=blue] at (1.5, .5){17}; 
\node[color=blue] at (2.5, .5){20}; 
\node  [color=blue] at (3.5, .5){23}; 
\node [color=blue] at (4.5, .5){25}; 
\node [color=blue] at (5.5, .5){27}; 
\node  [color=purple] at (6.5, .5){28}; 
\end{scope}

\begin{scope}[xshift=26cm]
\foreach \x in {0, 1,2,3,4, 5,6,7,8} {\draw[thick] (0,\x) -- (8, \x); }
\foreach \x in {0, 1,2,3,4,5,6,7,8} {\draw[thick] (\x,0) -- ( \x,8); }

\node  [color=purple] at (.5, 7.5){1}; 
\node  [color=cyan]at (1.5, 7.5){2}; 
\node [color=cyan] at (2.5, 7.5){3}; 
\node[color=cyan] at (3.5, 7.5){6}; 
\node[color=cyan] at (4.5, 7.5){8}; 
\node[color=cyan] at (5.5, 7.5){12}; 
\node[color=cyan] at (6.5, 7.5){15}; 
\node[color=cyan] at (7.5, 7.5){20}; 

\node  [color=blue] at (.5, 6.5){2}; 
\node   [color=purple] at (1.5, 6.5){4}; 
\node [color=cyan] at (2.5, 6.5){5}; 
\node[color=cyan] at (3.5, 6.5){7}; 
\node[color=cyan] at (4.5, 6.5){11}; 
\node[color=cyan] at (5.5, 6.5){14}; 
\node[color=cyan] at (6.5, 6.5){19}; 
\node[color=cyan] at (7.5, 6.5){23}; 

\node [color=blue] at (.5, 5.5){3}; 
\node [color=blue]  at (1.5, 5.5){6}; 
\node  [color=purple] at (2.5, 5.5){9}; 
\node[color=cyan] at (3.5, 5.5){10}; 
\node[color=cyan] at (4.5, 5.5){13}; 
\node[color=cyan] at (5.5, 5.5){18}; 
\node[color=cyan] at (6.5, 5.5){22}; 
\node[color=cyan] at (7.5, 5.5){27}; 

\node [color=blue] at (.5, 4.5){5};
\node   [color=blue]  at (1.5, 4.5){8}; 
\node  [color=blue] at (2.5, 4.5){12}; 
\node  [color=purple]  at (3.5, 4.5){16}; 
\node[color=cyan] at (4.5, 4.5){17}; 
\node[color=cyan] at (5.5, 4.5){21}; 
\node[color=cyan] at (6.5, 4.5){26}; 
\node[color=cyan] at (7.5, 4.5){29}; 

\node[color=blue] at (.5, 3.5){7}; 
\node [color=blue] at (1.5, 3.5){11}; 
\node [color=blue]  at (2.5, 3.5){15}; 
\node [color=blue] at (3.5, 3.5){20}; 
\node  [color=purple] at (4.5, 3.5){24}; 
\node[color=cyan] at (5.5, 3.5){25}; 
\node[color=cyan] at (6.5, 3.5){28}; 
\node[color=cyan] at (7.5, 3.5){32}; 

\node[color=blue] at (.5, 2.5){10}; 
\node[color=blue] at (1.5, 2.5){14}; 
\node [color=blue] at (2.5, 2.5){19}; 
\node [color=blue] at (3.5, 2.5){23}; 
\node [color=blue] at (4.5, 2.5){27}; 
\node  [color=purple] at (5.5, 2.5){30}; 
\node[color=cyan] at (6.5, 2.5){31}; 
\node[color=cyan] at (7.5, 2.5){33}; 

\node[color=blue] at (.5, 1.5){13}; 
\node[color=blue] at (1.5, 1.5){18}; 
\node[color=blue] at (2.5, 1.5){22}; 
\node  [color=blue] at (3.5, 1.5){26}; 
\node [color=blue] at (4.5, 1.5){29}; 
\node [color=blue] at (5.5, 1.5){32}; 
\node  [color=purple]  at (6.5, 1.5){34}; 
\node[color=cyan] at (7.5, 1.5){35}; 

\node[color=blue] at (.5, .5){17}; 
\node[color=blue] at (1.5, .5){21}; 
\node[color=blue] at (2.5, .5){25}; 
\node  [color=blue] at (3.5, .5){28}; 
\node [color=blue] at (4.5, .5){31}; 
\node [color=blue] at (5.5, .5){33}; 
\node [color=blue] at (6.5, .5){35}; 
\node  [color=purple] at (7.5, .5){36}; 
\end{scope}

\end{tikzpicture}
\caption{  Filling of the square when $r=3, 4, 5, 6, 7$ with matching diagonal. }
\label{figsqntc}
\end{figure}

When $\alpha=\beta$, we fill the squares by South East to North West diagonals with the terms in the lower half of each diagonal are matched with the terms in the upper half written in the same order(see Figure \ref{figsqntc}). 
Only the   middle square of these  diagonals have no matching terms.
Matching spots are at distance half the length of the diagonal.
At each stage, the only available spots to fill  are those left in the same diagonal and those in the next diagonal below those that have already been filled.
Matching spots are the only ones at distance half the length of the diagonal and the distance between them is even while the distance  between a spot on one diagonal and a spot on the next one is odd.
So, no additional empty spot could be filled with the same index.
As the process is symmetric (starting at the end and reversing the order of the indices would give the symmetric result), no already filled spot could be replaced with the same index.
So the filling is admissible.

When $\alpha<\beta$, the South East to North West diagonals selected for double filling are separated by a few spots in each column with single filling as shown in Figure \ref{fig:5colanglnum}.
In particular, before filling the last spot in the odd diagonals, we need to fill the column above this spot.
With an argument similar to the one above, this is also an admissible filling.

\begin{figure}[h!]
\begin{tikzpicture}[scale=.55]

\begin{scope}[]
\foreach \x in {0, 1,2,3,4,5} {\draw[thick] (0,\x) -- (5, \x); }
\foreach \x in {0, 1,2,3, 4, 5} {\draw[thick] (\x,0) -- ( \x,5); }
\node  [color=purple] at (.5, 4.5){1}; 
\node[color=cyan] at (1.5, 4.5){2}; 
\node[color=cyan] at (2.5, 4.5){3}; 
\node[color=cyan] at (3.5, 4.5){6}; 
\node[color=cyan] at (4.5, 4.5){8}; 

\node [color=blue]at (.5, 3.5){2}; 
\node[color=purple] at (1.5, 3.5){4}; 
\node[color=cyan] at (2.5, 3.5){5}; 
\node[color=cyan] at (3.5, 3.5){7}; 
\node[color=cyan] at (4.5, 3.5){11}; 

\node [color=blue] at (.5, 2.5){3}; 
\node  [color=blue] at (1.5, 2.5){6}; 
\node[color=purple] at (2.5, 2.5){9}; 
\node[color=cyan] at (3.5, 2.5){10}; 
\node[color=cyan] at (4.5, 2.5){12}; 

\node [color=blue]  at (.5, 1.5){5}; 
\node [color=blue] at (1.5, 1.5){8}; 
\node[color=blue] at (2.5, 1.5){11}; 
\node[color=purple] at (3.5, 1.5){13}; 
\node[color=cyan] at (4.5, 1.5){14}; 

\node[color=blue] at (.5, .5){7}; 
\node[color=blue] at (1.5, .5){10}; 
\node  [color=blue] at (2.5, .5){12}; 
\node[color=blue] at (3.5, .5){14}; 
\node[color=purple] at (4.5,.5){15}; 
\end{scope}

\begin{scope}[xshift=5.5cm]
\foreach \x in {0, 1,2,3,4,5,6} {\draw[thick] (0,\x) -- (5, \x); }
\foreach \x in {0, 1,2,3, 4, 5} {\draw[thick] (\x,0) -- ( \x,6); }

\node  [color=purple] at (.5, 5.5){1}; 
\node[color=cyan] at (1.5, 5.5){3}; 
\node[color=cyan] at (2.5, 5.5){4}; 
\node[color=cyan] at (3.5, 5.5){8}; 
\node[color=cyan] at (4.5, 5.5){10}; 

\node  [color=purple] at (.5, 4.5){2}; 
\node[color=purple] at (1.5, 4.5){5}; 
\node[color=cyan] at (2.5, 4.5){7}; 
\node[color=cyan] at (3.5, 4.5){9}; 
\node[color=cyan] at (4.5, 4.5){14}; 

\node [color=blue]at (.5, 3.5){3}; 
\node[color=purple] at (1.5, 3.5){6}; 
\node[color=purple] at (2.5, 3.5){11}; 
\node[color=cyan] at (3.5, 3.5){13}; 
\node[color=cyan] at (4.5, 3.5){15}; 

\node [color=blue] at (.5, 2.5){4}; 
\node  [color=blue] at (1.5, 2.5){8}; 
\node[color=purple] at (2.5, 2.5){12}; 
\node[color=purple] at (3.5, 2.5){16}; 
\node[color=cyan] at (4.5, 2.5){18}; 

\node [color=blue]  at (.5, 1.5){7}; 
\node [color=blue] at (1.5, 1.5){10}; 
\node[color=blue] at (2.5, 1.5){14}; 
\node[color=purple] at (3.5, 1.5){17}; 
\node[color=purple] at (4.5, 1.5){19}; 

\node[color=blue] at (.5, .5){9}; 
\node[color=blue] at (1.5, .5){13}; 
\node  [color=blue] at (2.5, .5){15}; 
\node[color=blue] at (3.5, .5){18}; 
\node[color=purple] at (4.5,.5){20}; 
\end{scope}

\begin{scope}[xshift=11cm]
\foreach \x in {0, 1,2,3,4,5,6,7} {\draw[thick] (0,\x) -- (5, \x); }
\foreach \x in {0, 1,2,3, 4, 5} {\draw[thick] (\x,0) -- ( \x,7); }
\node  [color=purple] at (.5, 6.5){1}; 
\node[color=cyan] at (1.5, 6.5){4}; 
\node[color=cyan] at (2.5, 6.5){7}; 
\node[color=cyan] at (3.5, 6.5){9}; 
\node[color=cyan] at (4.5, 6.5){13}; 

\node  [color=purple] at (.5, 5.5){2}; 
\node[color=cyan] at (1.5, 5.5){5}; 
\node[color=cyan] at (2.5, 5.5){8}; 
\node[color=cyan] at (3.5, 5.5){11}; 
\node[color=cyan] at (4.5, 5.5){15}; 

\node [color=purple]at (.5, 4.5){3}; 
\node[color=purple] at (1.5, 4.5){6}; 
\node[color=cyan] at (2.5, 4.5){10}; 
\node[color=cyan] at (3.5, 4.5){14}; 
\node[color=cyan] at (4.5, 4.5){18}; 

\node [color=blue] at (.5, 3.5){4}; 
\node  [color=blue] at (1.5, 3.5){7}; 
\node[color=purple] at (2.5, 3.5){12}; 
\node[color=cyan] at (3.5, 3.5){16}; 
\node[color=cyan] at (4.5, 3.5){19}; 

\node [color=blue] at (.5, 2.5){5}; 
\node  [color=blue] at (1.5, 2.5){9}; 
\node[color=blue] at (2.5, 2.5){13}; 
\node[color=purple] at (3.5, 2.5){17}; 
\node[color=purple] at (4.5, 2.5){20}; 

\node [color=blue]  at (.5, 1.5){8}; 
\node [color=blue] at (1.5, 1.5){11}; 
\node[color=blue] at (2.5, 1.5){15}; 
\node[color=blue] at (3.5, 1.5){18}; 
\node[color=purple] at (4.5, 1.5){21}; 

\node[color=blue] at (.5, .5){10}; 
\node[color=blue] at (1.5, .5){14}; 
\node  [color=blue] at (2.5, .5){16}; 
\node[color=blue] at (3.5, .5){19}; 
\node[color=purple] at (4.5,.5){22}; 
\end{scope}

\begin{scope}[xshift=16.5cm]
\foreach \x in {0, 1,2,3,4,5,6,7,8} {\draw[thick] (0,\x) -- (5, \x); }
\foreach \x in {0, 1,2,3, 4, 5} {\draw[thick] (\x,0) -- ( \x,8); }

\node  [color=purple] at (.5, 7.5){1}; 
\node[color=cyan] at (1.5, 7.5){5}; 
\node[color=cyan] at (2.5, 7.5){9}; 
\node[color=cyan] at (3.5, 7.5){11}; 
\node[color=cyan] at (4.5, 7.5){16}; 

\node  [color=purple] at (.5, 6.5){2}; 
\node[color=cyan] at (1.5, 6.5){6}; 
\node[color=cyan] at (2.5, 6.5){10}; 
\node[color=cyan] at (3.5, 6.5){13}; 
\node[color=cyan] at (4.5, 6.5){18}; 

\node [color=purple]at (.5, 5.5){3}; 
\node[color=purple] at (1.5, 5.5){7}; 
\node[color=cyan] at (2.5, 5.5){12}; 
\node[color=cyan] at (3.5, 5.5){17}; 
\node[color=cyan] at (4.5, 5.5){22}; 

\node [color=purple] at (.5, 4.5){4}; 
\node  [color=purple] at (1.5, 4.5){8}; 
\node[color=purple] at (2.5, 4.5){14}; 
\node[color=cyan] at (3.5, 4.5){19}; 
\node[color=cyan] at (4.5, 4.5){23}; 

\node [color=blue] at (.5, 3.5){5}; 
\node  [color=blue] at (1.5, 3.5){9}; 
\node[color=purple] at (2.5, 3.5){15}; 
\node[color=purple] at (3.5, 3.5){20}; 
\node[color=purple] at (4.5, 3.5){24}; 

\node [color=blue]  at (.5, 2.5){6}; 
\node [color=blue] at (1.5, 2.5){11}; 
\node[color=blue] at (2.5, 2.5){16}; 
\node[color=purple] at (3.5, 2.5){21}; 
\node[color=purple] at (4.5, 2.5){25}; 

\node [color=blue]  at (.5, 1.5){10}; 
\node [color=blue] at (1.5, 1.5){13}; 
\node[color=blue] at (2.5, 1.5){18}; 
\node[color=blue] at (3.5, 1.5){22}; 
\node[color=purple] at (4.5, 1.5){26}; 

\node[color=blue] at (.5, .5){12}; 
\node[color=blue] at (1.5, .5){17}; 
\node  [color=blue] at (2.5, .5){19}; 
\node[color=blue] at (3.5, .5){23}; 
\node[color=purple] at (4.5,.5){27}; 
\end{scope}

\begin{scope}[xshift=22cm]
\foreach \x in {0, 1,2,3,4,5,6,7,8,9} {\draw[thick] (0,\x) -- (5, \x); }
\foreach \x in {0, 1,2,3, 4, 5} {\draw[thick] (\x,0) -- ( \x,9); }
\node  [color=purple] at (.5, 8.5){1}; 
\node[color=cyan] at (1.5, 8.5){6}; 
\node[color=cyan] at (2.5, 8.5){11}; 
\node[color=cyan] at (3.5, 8.5){13}; 
\node[color=cyan] at (4.5, 8.5){19}; 

\node  [color=purple] at (.5, 7.5){2}; 
\node[color=cyan] at (1.5, 7.5){7}; 
\node[color=cyan] at (2.5, 7.5){12}; 
\node[color=cyan] at (3.5, 7.5){16}; 
\node[color=cyan] at (4.5, 7.5){21}; 

\node [color=purple]at (.5, 6.5){3}; 
\node[color=purple] at (1.5, 6.5){8}; 
\node[color=cyan] at (2.5, 6.5){14}; 
\node[color=cyan] at (3.5, 6.5){20}; 
\node[color=cyan] at (4.5, 6.5){26}; 

\node [color=purple] at (.5, 5.5){4}; 
\node  [color=purple] at (1.5, 5.5){9}; 
\node[color=purple] at (2.5, 5.5){16}; 
\node[color=cyan] at (3.5, 5.5){22}; 
\node[color=cyan] at (4.5, 5.5){27}; 

\node [color=purple] at (.5, 4.5){5}; 
\node  [color=purple] at (1.5, 4.5){10}; 
\node[color=purple] at (2.5, 4.5){17}; 
\node[color=purple] at (3.5, 4.5){23}; 
\node[color=purple] at (4.5, 4.5){28}; 

\node [color=blue] at (.5, 3.5){6}; 
\node  [color=blue] at (1.5, 3.5){11}; 
\node[color=purple] at (2.5, 3.5){18}; 
\node[color=purple] at (3.5, 3.5){24}; 
\node[color=purple] at (4.5, 3.5){29}; 

\node [color=blue]  at (.5, 2.5){7}; 
\node [color=blue] at (1.5, 2.5){13}; 
\node[color=blue] at (2.5, 2.5){19}; 
\node[color=purple] at (3.5, 2.5){25}; 
\node[color=purple] at (4.5, 2.5){30}; 

\node [color=blue]  at (.5, 1.5){12}; 
\node [color=blue] at (1.5, 1.5){15}; 
\node[color=blue] at (2.5, 1.5){21}; 
\node[color=blue] at (3.5, 1.5){26}; 
\node[color=purple] at (4.5, 1.5){31}; 

\node[color=blue] at (.5, .5){14}; 
\node[color=blue] at (1.5, .5){20}; 
\node  [color=blue] at (2.5, .5){22}; 
\node[color=blue] at (3.5, .5){27}; 
\node[color=purple] at (4.5,.5){32}; 
\end{scope}

\begin{scope}[xshift=27.5cm]
\foreach \x in {0, 1,2,3,4,5,6,7,8,9, 10} {\draw[thick] (0,\x) -- (5, \x); }
\foreach \x in {0, 1,2,3, 4, 5} {\draw[thick] (\x,0) -- ( \x,10); }
\node  [color=purple] at (.5, 9.5){1}; 
\node[color=cyan] at (1.5, 9.5){7}; 
\node[color=cyan] at (2.5, 9.5){13}; 
\node[color=cyan] at (3.5, 9.5){15}; 
\node[color=cyan] at (4.5, 9.5){22}; 

\node  [color=purple] at (.5, 8.5){2}; 
\node[color=cyan] at (1.5, 8.5){8}; 
\node[color=cyan] at (2.5, 8.5){14}; 
\node[color=cyan] at (3.5, 8.5){17}; 
\node[color=cyan] at (4.5, 8.5){24}; 

\node [color=purple]at (.5, 7.5){3}; 
\node[color=purple] at (1.5, 7.5){9}; 
\node[color=cyan] at (2.5, 7.5){16}; 
\node[color=cyan] at (3.5, 7.5){23}; 
\node[color=cyan] at (4.5, 7.5){30}; 

\node [color=purple]at (.5, 6.5){4}; 
\node[color=purple] at (1.5, 6.5){10}; 
\node[color=purple] at (2.5, 6.5){18}; 
\node[color=cyan] at (3.5, 6.5){25}; 
\node[color=cyan] at (4.5, 6.5){31}; 

\node [color=purple] at (.5, 5.5){5}; 
\node  [color=purple] at (1.5, 5.5){11}; 
\node[color=purple] at (2.5, 5.5){19}; 
\node[color=purple] at (3.5, 5.5){26}; 
\node[color=purple] at (4.5, 5.5){32}; 

\node [color=purple] at (.5, 4.5){6}; 
\node  [color=purple] at (1.5, 4.5){12}; 
\node[color=purple] at (2.5, 4.5){20}; 
\node[color=purple] at (3.5, 4.5){27}; 
\node[color=purple] at (4.5, 4.5){33}; 

\node [color=blue] at (.5, 3.5){7}; 
\node  [color=blue] at (1.5, 3.5){13}; 
\node[color=purple] at (2.5, 3.5){21}; 
\node[color=purple] at (3.5, 3.5){28}; 
\node[color=purple] at (4.5, 3.5){34}; 

\node [color=blue]  at (.5, 2.5){8}; 
\node [color=blue] at (1.5, 2.5){15}; 
\node[color=blue] at (2.5, 2.5){22}; 
\node[color=purple] at (3.5, 2.5){29}; 
\node[color=purple] at (4.5, 2.5){35}; 

\node [color=blue]  at (.5, 1.5){14}; 
\node [color=blue] at (1.5, 1.5){17}; 
\node[color=blue] at (2.5, 1.5){24}; 
\node[color=blue] at (3.5, 1.5){30}; 
\node[color=purple] at (4.5, 1.5){36}; 

\node[color=blue] at (.5, .5){16}; 
\node[color=blue] at (1.5, .5){23}; 
\node  [color=blue] at (2.5, .5){25}; 
\node[color=blue] at (3.5, .5){31}; 
\node[color=purple] at (4.5,.5){37}; 
\end{scope}

\end{tikzpicture}
\caption{ Filling a rectangle with 5 columns with matching corners.}
\label{fig:5colanglnum}
\end{figure}

\end{proof}

\begin{lemma}\label{lemfilex}  Let   $\alpha, \beta, g$ be integers such that  $\alpha\le \beta$, and $ \frac{\alpha(\beta+1)}2 \le g\le \alpha\beta$.
There exists  an admissible  filling of the $\alpha\times  \beta$ rectangle with the indices $1,\dots, g$ where $\alpha\beta-g$  of the indices appear exactly twice 
and the remainder appear exactly once.
\end{lemma}
\begin{proof} We construct admissible fillings corresponding to the smallest values of g so that at most $\alpha+1$ indices appear only once and the rest appear exactly twice.
For larger values of $g$, it suffices to replace the two spots where an index was used twice by consecutive indices and increase the remaining indices by one. 
One can repeat the process as necessary up to using $\alpha \beta=g$ indices when all the entries in the filling are different.


\begin{figure}[h!]
\begin{tikzpicture}[scale=.5]

\begin{scope}[]
\foreach \x in {0, 1,2,3,4} {\draw[thick] (0,\x) -- (4, \x); }
\foreach \x in {0, 1,2,3,4} {\draw[thick] (\x,0) -- ( \x,4); }

\node  [color=purple] at (.5, 3.5){1}; 
\node[color=cyan] at (1.5, 3.5){2}; 
\node[color=cyan] at (2.5, 3.5){3}; 
\node[color=cyan] at (3.5, 3.5){6}; 

\node [color=blue] at (.5, 2.5){2}; 
\node  [color=purple] at (1.5, 2.5){4}; 
\node[color=cyan] at (2.5, 2.5){5}; 
\node[color=cyan] at (3.5, 2.5){7}; 

\node [color=blue]  at (.5, 1.5){3}; 
\node [color=blue]  at (1.5, 1.5){6}; 
\node [color=purple] at (2.5, 1.5){8}; 
\node[color=cyan] at (3.5, 1.5){9}; 

\node[color=blue] at (.5, .5){5}; 
\node[color=blue] at (1.5, .5){7}; 
\node [color=blue]  at (2.5, .5){9}; 
\node[color=purple] at (3.5, .5){10}; 
\end{scope}

\begin{scope}[xshift=4.5cm]
\foreach \x in {0, 1,2,3,4,5} {\draw[thick] (0,\x) -- (4, \x); }
\foreach \x in {0, 1,2,3,4} {\draw[thick] (\x,0) -- ( \x,5); }

\node  [color=purple] at (.5, 4.5){1}; 
\node[color=cyan] at (1.5, 4.5){2}; 
\node[color=cyan] at (2.5, 4.5){3}; 
\node[color=cyan] at (3.5, 4.5){6};

\node [color=blue]at (.5, 3.5){2}; 
\node[color=purple] at (1.5, 3.5){4}; 
\node[color=cyan] at (2.5, 3.5){5}; 
\node[color=cyan] at (3.5, 3.5){8};

\node [color=blue] at (.5, 2.5){3}; 
\node  [color=blue] at (1.5, 2.5){6}; 
\node[color=cyan] at (2.5, 2.5){7}; 
\node[color=cyan] at (3.5, 2.5){9};

\node [color=blue]  at (.5, 1.5){5}; 
\node [color=blue] at (1.5, 1.5){8}; 
\node[color=purple] at (2.5, 1.5){10}; 
\node[color=cyan] at (3.5, 1.5){11};

\node[color=blue] at (.5, .5){7}; 
\node[color=blue] at (1.5, .5){9}; 
\node  [color=blue] at (2.5, .5){11}; 
\node[color=purple] at (3.5, .5){12};
\end{scope}

\begin{scope}[xshift=9cm]
\foreach \x in {0, 1,2,3,4,5,6} {\draw[thick] (0,\x) -- (4, \x); }
\foreach \x in {0, 1,2,3,4} {\draw[thick] (\x,0) -- ( \x,6); }

\node  [color=purple] at (.5, 5.5){1}; 
\node[color=cyan] at (1.5, 5.5){2}; 
\node[color=cyan] at (2.5, 5.5){3}; 
\node[color=cyan] at (3.5, 5.5){6};

\node [color=blue]at (.5, 4.5){2}; 
\node[color=purple] at (1.5, 4.5){4}; 
\node[color=cyan] at (2.5, 4.5){5}; 
\node[color=cyan] at (3.5, 4.5){8};

\node [color=blue] at (.5, 3.5){3}; 
\node  [color=blue] at (1.5, 3.5){6}; 
\node[color=cyan] at (2.5, 3.5){7}; 
\node[color=cyan] at (3.5, 3.5){10};

\node [color=blue]  at (.5, 2.5){5}; 
\node [color=blue] at (1.5, 2.5){8}; 
\node[color=cyan] at (2.5, 2.5){9}; 
\node[color=cyan] at (3.5, 2.5){11};

\node[color=blue] at (.5,1 .5){7}; 
\node[color=blue] at (1.5, 1.5){10}; 
\node  [color=purple] at (2.5, 1.5){12}; 
\node[color=blue] at (3.5, 1.5){13};

\node[color=blue] at (.5, .5){9}; 
\node[color=blue] at (1.5, .5){11}; 
\node  [color=blue] at (2.5, .5){13}; 
\node[color=purple] at (3.5, .5){14};
\end{scope}

\begin{scope}[xshift=13.5cm]
\foreach \x in {0, 1,2,3,4,5,6,7} {\draw[thick] (0,\x) -- (4, \x); }
\foreach \x in {0, 1,2,3,4} {\draw[thick] (\x,0) -- ( \x,7); }

\node  [color=purple] at (.5, 6.5){1}; 
\node[color=cyan] at (1.5, 6.5){2}; 
\node[color=cyan] at (2.5, 6.5){3}; 
\node[color=cyan] at (3.5, 6.5){6};

\node [color=blue]at (.5, 5.5){2}; 
\node[color=purple] at (1.5, 5.5){4}; 
\node[color=cyan] at (2.5, 5.5){5}; 
\node[color=cyan] at (3.5, 5.5){8};

\node [color=blue] at (.5, 4.5){3}; 
\node  [color=blue] at (1.5, 4.5){6}; 
\node[color=cyan] at (2.5, 4.5){7}; 
\node[color=cyan] at (3.5, 4.5){10};

\node [color=blue]  at (.5, 3.5){5}; 
\node [color=blue] at (1.5, 3.5){8}; 
\node[color=cyan] at (2.5, 3.5){9}; 
\node[color=cyan] at (3.5, 3.5){12};

\node[color=blue] at (.5,2 .5){7}; 
\node[color=blue] at (1.5, 2.5){10}; 
\node  [color=cyan] at (2.5, 2.5){11}; 
\node[color=cyan] at (3.5, 2.5){13};

\node[color=blue] at (.5,1 .5){9}; 
\node[color=blue] at (1.5, 1.5){12}; 
\node  [color=purple] at (2.5, 1.5){14}; 
\node[color=cyan] at (3.5, 1.5){15};

\node[color=blue] at (.5, .5){11}; 
\node[color=blue] at (1.5, .5){13}; 
\node  [color=blue] at (2.5, .5){15}; 
\node[color=purple] at (3.5, .5){16};
\end{scope}

\begin{scope}[xshift=18cm]
\foreach \x in {0, 1,2,3,4,5,6,7,8} {\draw[thick] (0,\x) -- (4, \x); }
\foreach \x in {0, 1,2,3,4} {\draw[thick] (\x,0) -- ( \x,8); }

\node  [color=purple] at (.5, 7.5){1}; 
\node[color=cyan] at (1.5, 7.5){2}; 
\node[color=cyan] at (2.5, 7.5){3}; 
\node[color=cyan] at (3.5, 7.5){6};

\node [color=blue]at (.5, 6.5){2}; 
\node[color=purple] at (1.5, 6.5){4}; 
\node[color=cyan] at (2.5, 6.5){5}; 
\node[color=cyan] at (3.5, 6.5){8};

\node [color=blue] at (.5, 5.5){3}; 
\node  [color=blue] at (1.5, 5.5){6}; 
\node[color=cyan] at (2.5, 5.5){7}; 
\node[color=cyan] at (3.5, 5.5){10};

\node [color=blue]  at (.5, 4.5){5}; 
\node [color=blue] at (1.5, 4.5){8}; 
\node[color=cyan] at (2.5, 4.5){9}; 
\node[color=cyan] at (3.5, 4.5){12};

\node[color=blue] at (.5,3 .5){7}; 
\node[color=blue] at (1.5, 3.5){10}; 
\node  [color=cyan] at (2.5, 3.5){11}; 
\node[color=cyan] at (3.5, 3.5){14};

\node[color=blue] at (.5,2 .5){9}; 
\node[color=blue] at (1.5, 2.5){12}; 
\node  [color=cyan] at (2.5, 2.5){13}; 
\node[color=cyan] at (3.5, 2.5){15};

\node[color=blue] at (.5,1 .5){11}; 
\node[color=blue] at (1.5, 1.5){14}; 
\node  [color=purple] at (2.5, 1.5){16}; 
\node[color=cyan] at (3.5, 1.5){17};

\node[color=blue] at (.5, .5){13}; 
\node[color=blue] at (1.5, .5){15}; 
\node  [color=blue] at (2.5, .5){17}; 
\node[color=purple] at (3.5, .5){18};
\end{scope}

\begin{scope}[xshift=22.5cm]
\foreach \x in {0, 1,2,3,4,5,6,7,8,9} {\draw[thick] (0,\x) -- (4, \x); }
\foreach \x in {0, 1,2,3,4} {\draw[thick] (\x,0) -- ( \x,9); }

\node  [color=purple] at (.5, 8.5){1}; 
\node[color=cyan] at (1.5, 8.5){2}; 
\node[color=cyan] at (2.5, 8.5){3}; 
\node[color=cyan] at (3.5, 8.5){6};

\node [color=blue]at (.5, 7.5){2}; 
\node[color=purple] at (1.5, 7.5){4}; 
\node[color=cyan] at (2.5, 7.5){5}; 
\node[color=cyan] at (3.5, 7.5){8};

\node [color=blue] at (.5, 6.5){3}; 
\node  [color=blue] at (1.5, 6.5){6}; 
\node[color=cyan] at (2.5, 6.5){7}; 
\node[color=cyan] at (3.5, 6.5){10};

\node [color=blue]  at (.5, 5.5){5}; 
\node [color=blue] at (1.5, 5.5){8}; 
\node[color=cyan] at (2.5, 5.5){9}; 
\node[color=cyan] at (3.5, 5.5){12};

\node[color=blue] at (.5,4 .5){7}; 
\node[color=blue] at (1.5, 4.5){10}; 
\node  [color=cyan] at (2.5, 4.5){11}; 
\node[color=cyan] at (3.5, 4.5){14};

\node[color=blue] at (.5,3.5){9}; 
\node[color=blue] at (1.5, 3.5){12}; 
\node  [color=cyan] at (2.5, 3.5){13}; 
\node[color=cyan] at (3.5, 3.5){16};

\node[color=blue] at (.5,2 .5){11}; 
\node[color=blue] at (1.5, 2.5){14}; 
\node  [color=cyan] at (2.5, 2.5){15}; 
\node[color=cyan] at (3.5, 2.5){17};

\node[color=blue] at (.5,1 .5){13}; 
\node[color=blue] at (1.5, 1.5){16}; 
\node  [color=purple] at (2.5, 1.5){18}; 
\node[color=cyan] at (3.5, 1.5){19};

\node[color=blue] at (.5, .5){15}; 
\node[color=blue] at (1.5, .5){17}; 
\node  [color=blue] at (2.5, .5){19}; 
\node[color=purple] at (3.5, .5){20};
\end{scope}

\end{tikzpicture}
\caption{Admissible fillings when $\alpha=4$. .}
\label{fig:alpha=4}
\end{figure}
When $\alpha$ is even, we fill by diagonals (see Figure \ref{fig:alpha=4}  for $\alpha=4$). On the middle spot of the odd diagonals at the start and end, we use non-repeated indices,
so  that there are in total $\alpha$ non-repeated indices.

When $\alpha$ is odd, we fill the first column of the rectangle until there are at most $\alpha -1$ spots left in this column. We match the  elements in the first column with elements on the top rows.
Then we fill diagonally as in the even case. 
We use any spots left over in the first column to match with the center of the diagonals of odd length.
When filling the diagonal that starts on the last filled spot of the first column, in order to get an admissible filling, we may need to use an additional non-matching index. 
Then we use non-matching indices for the center spots of the decreasing diagonals at the end (see figures \ref{fig:3col}, \ref{fig:5col} for $\alpha=3, 5$).

\begin{figure}[h!]
\begin{tikzpicture}[scale=.55]

\begin{scope}
\foreach \x in {0, 1,2,3} {\draw[thick] (0,\x) -- (3, \x); }
\foreach \x in {0, 1,2,3} {\draw[thick] (\x,0) -- ( \x,3); }

\node  [color=purple] at (.5, 2.5){1}; 
\node[color=cyan] at (1.5, 2.5){2}; 
\node[color=cyan] at (2.5, 2.5){3}; 

\node [color=blue]  at (.5, 1.5){2}; 
\node  [color=purple] at (1.5, 1.5){4}; 
\node[color=cyan] at (2.5, 1.5){5}; 

\node[color=blue] at (.5, .5){3}; 
\node[color=blue] at (1.5, .5){5}; 
\node  [color=purple] at (2.5, .5){6}; 
\end{scope}

\begin{scope}[xshift=3.5cm]
\foreach \x in {0, 1,2,3,4} {\draw[thick] (0,\x) -- (3, \x); }
\foreach \x in {0, 1,2,3} {\draw[thick] (\x,0) -- ( \x,4); }

\node  [color=purple] at (.5, 3.5){1}; 
\node[color=cyan] at (1.5, 3.5){2}; 
\node[color=cyan] at (2.5, 3.5){3}; 

\node [color=blue] at (.5, 2.5){2}; 
\node  [color=cyan] at (1.5, 2.5){4}; 
\node[color=purple] at (2.5, 2.5){6}; 

\node [color=blue]  at (.5, 1.5){3}; 
\node [color=purple]  at (1.5, 1.5){5}; 
\node[color=cyan] at (2.5, 1.5){7}; 

\node[color=blue] at (.5, .5){4}; 
\node[color=blue] at (1.5, .5){7}; 
\node [color=purple]  at (2.5, .5){8}; 
\end{scope}

\begin{scope}[xshift=7cm]
\foreach \x in {0, 1,2,3,4,5} {\draw[thick] (0,\x) -- (3, \x); }
\foreach \x in {0, 1,2,3} {\draw[thick] (\x,0) -- ( \x,5); }

\node  [color=purple] at (.5, 4.5){1}; 
\node[color=cyan] at (1.5, 4.5){2}; 
\node[color=cyan] at (2.5, 4.5){3}; 

\node [color=blue]at (.5, 3.5){2}; 
\node[color=cyan] at (1.5, 3.5){4}; 
\node[color=cyan] at (2.5, 3.5){6}; 

\node [color=blue] at (.5, 2.5){3}; 
\node  [color=purple] at (1.5, 2.5){5}; 
\node[color=cyan] at (2.5, 2.5){7}; 

\node [color=blue]  at (.5, 1.5){4}; 
\node [color=blue] at (1.5, 1.5){6}; 
\node[color=cyan] at (2.5, 1.5){8}; 

\node[color=blue] at (.5, .5){7}; 
\node[color=blue] at (1.5, .5){8}; 
\node  [color=purple] at (2.5, .5){9}; 
\end{scope}

\begin{scope}[xshift=10.5cm]
\foreach \x in {0, 1,2,3,4,5,6} {\draw[thick] (0,\x) -- (3, \x); }
\foreach \x in {0, 1,2,3} {\draw[thick] (\x,0) -- ( \x,6); }

\node  [color=purple] at (.5, 5.5){1}; 
\node[color=cyan] at (1.5, 5.5){2}; 
\node[color=cyan] at (2.5, 5.5){3}; 

\node  [color=blue] at (.5, 4.5){2}; 
\node[color=cyan] at (1.5, 4.5){4}; 
\node[color=cyan] at (2.5, 4.5){5}; 

\node [color=blue] at (.5, 3.5){3}; 
\node  [color=blue] at (1.5, 3.5){6}; 
\node[color=cyan] at (2.5, 3.5){7}; 

\node [color=blue] at (.5, 2.5){4}; 
\node  [color=blue] at (1.5, 2.5){7}; 
\node[color=purple] at (2.5, 2.5){9}; 

\node [color=blue]  at (.5, 1.5){5}; 
\node [color=purple] at (1.5, 1.5){8}; 
\node[color=cyan] at (2.5, 1.5){10}; 

\node[color=blue] at (.5, .5){6}; 
\node[color=blue] at (1.5, .5){10}; 
\node  [color=purple] at (2.5, .5){11}; 
\end{scope}

\begin{scope}[xshift=14cm]
\foreach \x in {0, 1,2,3,4,5,6,7} {\draw[thick] (0,\x) -- (3, \x); }
\foreach \x in {0, 1,2,3} {\draw[thick] (\x,0) -- ( \x,7); }

\node  [color=purple] at (.5, 6.5){1}; 
\node[color=cyan] at (1.5, 6.5){2}; 
\node[color=cyan] at (2.5, 6.5){3}; 

\node  [color=blue] at (.5, 5.5){2}; 
\node[color=cyan] at (1.5, 5.5){4}; 
\node[color=cyan] at (2.5, 5.5){5}; 

\node  [color=blue] at (.5, 4.5){3}; 
\node[color=cyan] at (1.5, 4.5){6}; 
\node[color=cyan] at (2.5, 4.5){7}; 

\node [color=blue] at (.5, 3.5){4}; 
\node  [color=blue] at (1.5, 3.5){7}; 
\node[color=cyan] at (2.5, 3.5){9}; 

\node [color=blue] at (.5, 2.5){5}; 
\node  [color=purple] at (1.5, 2.5){8}; 
\node[color=cyan] at (2.5, 2.5){10}; 

\node [color=blue]  at (.5, 1.5){6}; 
\node [color=blue] at (1.5, 1.5){9}; 
\node[color=cyan] at (2.5, 1.5){11}; 

\node[color=blue] at (.5, .5){10}; 
\node[color=blue] at (1.5, .5){11}; 
\node  [color=purple] at (2.5, .5){12}; 
\end{scope}

\begin{scope}[xshift=17.5cm]
\foreach \x in {0, 1,2,3,4,5,6,7,8} {\draw[thick] (0,\x) -- (3, \x); }
\foreach \x in {0, 1,2,3} {\draw[thick] (\x,0) -- ( \x,8); }

\node  [color=purple] at (.5, 7.5){1}; 
\node[color=cyan] at (1.5, 7.5){2}; 
\node[color=cyan] at (2.5, 7.5){3}; 

\node [color=blue] at (.5, 6.5){2}; 
\node[color=cyan] at (1.5, 6.5){4}; 
\node[color=cyan] at (2.5, 6.5){5}; 

\node  [color=blue] at (.5, 5.5){3}; 
\node[color=cyan] at (1.5, 5.5){6}; 
\node[color=cyan] at (2.5, 5.5){7}; 

\node  [color=blue] at (.5, 4.5){4}; 
\node[color=cyan] at (1.5, 4.5){8}; 
\node[color=cyan] at (2.5, 4.5){9}; 

\node [color=blue] at (.5, 3.5){5}; 
\node  [color=blue]  at (1.5, 3.5){9}; 
\node[color=cyan] at (2.5, 3.5){10}; 

\node [color=blue] at (.5, 2.5){6}; 
\node  [color=blue] at (1.5, 2.5){10}; 
\node[color=purple] at (2.5, 2.5){12}; 

\node [color=blue]  at (.5, 1.5){7}; 
\node [color=purple] at (1.5, 1.5){11}; 
\node[color=cyan] at (2.5, 1.5){13}; 

\node[color=blue] at (.5, .5){8}; 
\node[color=blue] at (1.5, .5){13}; 
\node  [color=purple] at (2.5, .5){14}; 
\end{scope}

\begin{scope}[xshift=21cm]
\foreach \x in {0, 1,2,3,4,5,6,7,8,9} {\draw[thick] (0,\x) -- (3, \x); }
\foreach \x in {0, 1,2,3} {\draw[thick] (\x,0) -- ( \x,9); }
\node  [color=purple] at (.5, 8.5){1}; 
\node[color=cyan] at (1.5, 8.5){2}; 
\node[color=cyan] at (2.5, 8.5){3}; 

\node [color=blue] at (.5, 7.5){2}; 
\node[color=cyan] at (1.5, 7.5){4}; 
\node[color=cyan] at (2.5, 7.5){5}; 

\node [color=blue] at (.5, 6.5){3}; 
\node[color=cyan] at (1.5, 6.5){6}; 
\node[color=cyan] at (2.5, 6.5){7}; 

\node  [color=blue] at (.5, 5.5){4}; 
\node[color=cyan] at (1.5, 5.5){8}; 
\node[color=cyan] at (2.5, 5.5){9}; 

\node  [color=blue] at (.5, 4.5){5}; 
\node[color=cyan] at (1.5, 4.5){9}; 
\node[color=cyan] at (2.5, 4.5){10}; 

\node [color=blue] at (.5, 3.5){6}; 
\node  [color=blue]  at (1.5, 3.5){10}; 
\node[color=cyan] at (2.5, 3.5){12}; 

\node [color=blue] at (.5, 2.5){7}; 
\node  [color=purple] at (1.5, 2.5){11}; 
\node[color=cyan] at (2.5, 2.5){13}; 

\node [color=blue]  at (.5, 1.5){8}; 
\node [color=blue] at (1.5, 1.5){12}; 
\node[color=cyan] at (2.5, 1.5){14}; 

\node[color=blue] at (.5, .5){13}; 
\node[color=blue] at (1.5, .5){14}; 
\node  [color=purple] at (2.5, .5){15}; 
\end{scope}

\end{tikzpicture}
\caption{ Admissible fillings  for $\alpha=3$.}
\label{fig:3col}
\end{figure}

\begin{figure}[h!]
\begin{tikzpicture}[scale=.55]

\begin{scope}[]
\foreach \x in {0, 1,2,3,4,5} {\draw[thick] (0,\x) -- (5, \x); }
\foreach \x in {0, 1,2,3, 4, 5} {\draw[thick] (\x,0) -- ( \x,5); }
\node  [color=purple] at (.5, 4.5){1}; 
\node[color=cyan] at (1.5, 4.5){2}; 
\node[color=cyan] at (2.5, 4.5){3}; 
\node[color=cyan] at (3.5, 4.5){4}; 
\node[color=cyan] at (4.5, 4.5){5}; 

\node [color=blue]at (.5, 3.5){2}; 
\node[color=purple] at (1.5, 3.5){6}; 
\node[color=cyan] at (2.5, 3.5){7}; 
\node[color=cyan] at (3.5, 3.5){8}; 
\node[color=cyan] at (4.5, 3.5){11}; 

\node [color=blue] at (.5, 2.5){3}; 
\node  [color=blue] at (1.5, 2.5){7}; 
\node[color=purple] at (2.5, 2.5){9}; 
\node[color=cyan] at (3.5, 2.5){10}; 
\node[color=cyan] at (4.5, 2.5){12}; 

\node [color=blue]  at (.5, 1.5){4}; 
\node [color=blue] at (1.5, 1.5){8}; 
\node[color=blue] at (2.5, 1.5){11}; 
\node[color=purple] at (3.5, 1.5){13}; 
\node[color=cyan] at (4.5, 1.5){14}; 

\node[color=blue] at (.5, .5){5}; 
\node[color=blue] at (1.5, .5){10}; 
\node  [color=blue] at (2.5, .5){12}; 
\node[color=blue] at (3.5, .5){14}; 
\node[color=purple] at (4.5,.5){15}; 
\end{scope}

\begin{scope}[xshift=5.5cm]
\foreach \x in {0, 1,2,3,4,5,6} {\draw[thick] (0,\x) -- (5, \x); }
\foreach \x in {0, 1,2,3, 4, 5} {\draw[thick] (\x,0) -- ( \x,6); }

\node  [color=purple] at (.5, 5.5){1}; 
\node[color=cyan] at (1.5, 5.5){2}; 
\node[color=cyan] at (2.5, 5.5){3}; 
\node[color=cyan] at (3.5, 5.5){4}; 
\node[color=cyan] at (4.5, 5.5){5}; 

\node  [color=blue] at (.5, 4.5){2}; 
\node[color=cyan] at (1.5, 4.5){6}; 
\node[color=cyan] at (2.5, 4.5){7}; 
\node[color=cyan] at (3.5, 4.5){8}; 
\node[color=purple] at (4.5, 4.5){12}; 

\node [color=blue]at (.5, 3.5){3}; 
\node[color=blue] at (1.5, 3.5){7}; 
\node[color=purple] at (2.5, 3.5){9}; 
\node[color=cyan] at (3.5, 3.5){10}; 
\node[color=cyan] at (4.5, 3.5){14}; 

\node [color=blue] at (.5, 2.5){4}; 
\node  [color=blue] at (1.5, 2.5){8}; 
\node[color=purple] at (2.5, 2.5){11}; 
\node[color=cyan] at (3.5, 2.5){13}; 
\node[color=cyan] at (4.5, 2.5){15}; 

\node [color=blue]  at (.5, 1.5){5}; 
\node [color=blue] at (1.5, 1.5){10}; 
\node[color=blue] at (2.5, 1.5){14}; 
\node[color=purple] at (3.5, 1.5){16}; 
\node[color=cyan] at (4.5, 1.5){17}; 

\node[color=blue] at (.5, .5){6}; 
\node[color=blue] at (1.5, .5){13}; 
\node  [color=blue] at (2.5, .5){15}; 
\node[color=blue] at (3.5, .5){17}; 
\node[color=purple] at (4.5,.5){18}; 
\end{scope}

\begin{scope}[xshift=11cm]
\foreach \x in {0, 1,2,3,4,5,6,7} {\draw[thick] (0,\x) -- (5, \x); }
\foreach \x in {0, 1,2,3, 4, 5} {\draw[thick] (\x,0) -- ( \x,7); }
\node  [color=purple] at (.5, 6.5){1}; 
\node[color=cyan] at (1.5, 6.5){2}; 
\node[color=cyan] at (2.5, 6.5){3}; 
\node[color=cyan] at (3.5, 6.5){4}; 
\node[color=cyan] at (4.5, 6.5){5}; 

\node  [color=blue] at (.5, 5.5){2}; 
\node[color=cyan] at (1.5, 5.5){6}; 
\node[color=cyan] at (2.5, 5.5){7}; 
\node[color=cyan] at (3.5, 5.5){8}; 
\node[color=cyan] at (4.5, 5.5){11}; 

\node [color=blue]at (.5, 4.5){3}; 
\node[color=blue] at (1.5, 4.5){7}; 
\node[color=cyan] at (2.5, 4.5){9}; 
\node[color=cyan] at (3.5, 4.5){10}; 
\node[color=purple] at (4.5, 4.5){14}; 

\node [color=blue] at (.5, 3.5){4}; 
\node  [color=blue] at (1.5, 3.5){8}; 
\node[color=blue] at (2.5, 3.5){11}; 
\node[color=cyan] at (3.5, 3.5){12}; 
\node[color=cyan] at (4.5, 3.5){16}; 

\node [color=blue] at (.5, 2.5){5}; 
\node  [color=blue] at (1.5, 2.5){10}; 
\node[color=purple] at (2.5, 2.5){13}; 
\node[color=cyan] at (3.5, 2.5){15}; 
\node[color=cyan] at (4.5, 2.5){17}; 

\node [color=blue]  at (.5, 1.5){6}; 
\node [color=blue] at (1.5, 1.5){12}; 
\node[color=blue] at (2.5, 1.5){16}; 
\node[color=purple] at (3.5, 1.5){18}; 
\node[color=cyan] at (4.5, 1.5){19}; 

\node[color=blue] at (.5, .5){9}; 
\node[color=blue] at (1.5, .5){15}; 
\node  [color=blue] at (2.5, .5){17}; 
\node[color=blue] at (3.5, .5){19}; 
\node[color=purple] at (4.5,.5){20}; 
\end{scope}

\begin{scope}[xshift=11cm]
\foreach \x in {0, 1,2,3,4,5,6,7} {\draw[thick] (0,\x) -- (5, \x); }
\foreach \x in {0, 1,2,3, 4, 5} {\draw[thick] (\x,0) -- ( \x,7); }

\node  [color=purple] at (.5, 6.5){1}; 
\node[color=cyan] at (1.5, 6.5){2}; 
\node[color=cyan] at (2.5, 6.5){3}; 
\node[color=cyan] at (3.5, 6.5){4}; 
\node[color=cyan] at (4.5, 6.5){5}; 

\node  [color=blue] at (.5, 5.5){2}; 
\node[color=cyan] at (1.5, 5.5){6}; 
\node[color=cyan] at (2.5, 5.5){7}; 
\node[color=cyan] at (3.5, 5.5){8}; 
\node[color=cyan] at (4.5, 5.5){11}; 

\node [color=blue]at (.5, 4.5){3}; 
\node[color=blue] at (1.5, 4.5){7}; 
\node[color=cyan] at (2.5, 4.5){9}; 
\node[color=cyan] at (3.5, 4.5){10}; 
\node[color=purple] at (4.5, 4.5){14}; 

\node [color=blue] at (.5, 3.5){4}; 
\node  [color=blue] at (1.5, 3.5){8}; 
\node[color=blue] at (2.5, 3.5){11}; 
\node[color=cyan] at (3.5, 3.5){12}; 
\node[color=cyan] at (4.5, 3.5){16}; 

\node [color=blue] at (.5, 2.5){5}; 
\node  [color=blue] at (1.5, 2.5){10}; 
\node[color=purple] at (2.5, 2.5){13}; 
\node[color=cyan] at (3.5, 2.5){15}; 
\node[color=cyan] at (4.5, 2.5){17}; 

\node [color=blue]  at (.5, 1.5){6}; 
\node [color=blue] at (1.5, 1.5){12}; 
\node[color=blue] at (2.5, 1.5){16}; 
\node[color=purple] at (3.5, 1.5){18}; 
\node[color=cyan] at (4.5, 1.5){19}; 

\node[color=blue] at (.5, .5){9}; 
\node[color=blue] at (1.5, .5){15}; 
\node  [color=blue] at (2.5, .5){17}; 
\node[color=blue] at (3.5, .5){19}; 
\node[color=purple] at (4.5,.5){20}; 
\end{scope}

\begin{scope}[xshift=16.5cm]
\foreach \x in {0, 1,2,3,4,5,6,7,8} {\draw[thick] (0,\x) -- (5, \x); }
\foreach \x in {0, 1,2,3, 4, 5} {\draw[thick] (\x,0) -- ( \x,8); }

\node  [color=purple] at (.5, 7.5){1}; 
\node[color=cyan] at (1.5, 7.5){2}; 
\node[color=cyan] at (2.5, 7.5){3}; 
\node[color=cyan] at (3.5, 7.5){4}; 
\node[color=cyan] at (4.5, 7.5){5}; 

\node  [color=blue] at (.5, 6.5){2}; 
\node[color=cyan] at (1.5, 6.5){6}; 
\node[color=cyan] at (2.5, 6.5){7}; 
\node[color=cyan] at (3.5, 6.5){8}; 
\node[color=cyan] at (4.5, 6.5){11}; 

\node [color=blue]at (.5, 5.5){3}; 
\node[color=blue] at (1.5, 5.5){7}; 
\node[color=cyan] at (2.5, 5.5){9}; 
\node[color=cyan] at (3.5, 5.5){10}; 
\node[color=cyan] at (4.5, 5.5){14}; 

\node [color=blue] at (.5, 4.5){4}; 
\node  [color=blue] at (1.5, 4.5){8}; 
\node[color=blue] at (2.5, 4.5){11}; 
\node[color=cyan] at (3.5, 4.5){12}; 
\node[color=purple] at (4.5, 4.5){17}; 

\node [color=blue] at (.5, 3.5){5}; 
\node  [color=blue] at (1.5, 3.5){10}; 
\node[color=purple] at (2.5, 3.5){13}; 
\node[color=cyan] at (3.5, 3.5){15}; 
\node[color=cyan] at (4.5, 3.5){19}; 

\node [color=blue]  at (.5, 2.5){6}; 
\node [color=blue] at (1.5, 2.5){12}; 
\node[color=purple] at (2.5, 2.5){16}; 
\node[color=cyan] at (3.5, 2.5){18}; 
\node[color=cyan] at (4.5, 2.5){20}; 

\node [color=blue]  at (.5, 1.5){9}; 
\node [color=blue] at (1.5, 1.5){15}; 
\node[color=blue] at (2.5, 1.5){19}; 
\node[color=purple] at (3.5, 1.5){21}; 
\node[color=cyan] at (4.5, 1.5){22}; 

\node[color=blue] at (.5, .5){14}; 
\node[color=blue] at (1.5, .5){18}; 
\node  [color=blue] at (2.5, .5){20}; 
\node[color=blue] at (3.5, .5){22}; 
\node[color=purple] at (4.5,.5){23}; 
\end{scope}

\begin{scope}[xshift=22cm]
\foreach \x in {0, 1,2,3,4,5,6,7,8,9} {\draw[thick] (0,\x) -- (5, \x); }
\foreach \x in {0, 1,2,3, 4, 5} {\draw[thick] (\x,0) -- ( \x,9); }
\node  [color=purple] at (.5, 8.5){1}; 
\node[color=cyan] at (1.5, 8.5){2}; 
\node[color=cyan] at (2.5, 8.5){3}; 
\node[color=cyan] at (3.5, 8.5){4}; 
\node[color=cyan] at (4.5, 8.5){5}; 

\node  [color=blue] at (.5, 7.5){2}; 
\node[color=cyan] at (1.5, 7.5){6}; 
\node[color=cyan] at (2.5, 7.5){7}; 
\node[color=cyan] at (3.5, 7.5){8}; 
\node[color=cyan] at (4.5, 7.5){11}; 

\node [color=blue]at (.5, 6.5){3}; 
\node[color=blue] at (1.5, 6.5){7}; 
\node[color=cyan] at (2.5, 6.5){9}; 
\node[color=cyan] at (3.5, 6.5){10}; 
\node[color=cyan] at (4.5, 6.5){12}; 

\node [color=blue] at (.5, 5.5){4}; 
\node  [color=blue] at (1.5, 5.5){8}; 
\node[color=blue] at (2.5, 5.5){11}; 
\node[color=cyan] at (3.5, 5.5){13}; 
\node[color=cyan] at (4.5, 5.5){16}; 

\node [color=blue] at (.5, 4.5){5}; 
\node  [color=blue] at (1.5, 4.5){10}; 
\node[color=cyan] at (2.5, 4.5){14}; 
\node[color=cyan] at (3.5, 4.5){15}; 
\node[color=purple] at (4.5, 4.5){19}; 

\node [color=blue] at (.5, 3.5){6}; 
\node  [color=blue] at (1.5, 3.5){13}; 
\node[color=blue] at (2.5, 3.5){16}; 
\node[color=cyan] at (3.5, 3.5){17}; 
\node[color=cyan] at (4.5, 3.5){21}; 

\node [color=blue]  at (.5, 2.5){9}; 
\node [color=blue] at (1.5, 2.5){15}; 
\node[color=purple] at (2.5, 2.5){18}; 
\node[color=cyan] at (3.5, 2.5){20}; 
\node[color=cyan] at (4.5, 2.5){22}; 

\node [color=blue]  at (.5, 1.5){12}; 
\node [color=blue] at (1.5, 1.5){17}; 
\node[color=blue] at (2.5, 1.5){21}; 
\node[color=purple] at (3.5, 1.5){23}; 
\node[color=cyan] at (4.5, 1.5){24}; 

\node[color=blue] at (.5, .5){14}; 
\node[color=blue] at (1.5, .5){20}; 
\node  [color=blue] at (2.5, .5){22}; 
\node[color=blue] at (3.5, .5){24}; 
\node[color=purple] at (4.5,.5){25}; 
\end{scope}

\begin{scope}[xshift=27.5cm]
\foreach \x in {0, 1,2,3,4,5,6,7,8,9, 10} {\draw[thick] (0,\x) -- (5, \x); }
\foreach \x in {0, 1,2,3, 4, 5} {\draw[thick] (\x,0) -- ( \x,10); }
\node  [color=purple] at (.5, 9.5){1}; 
\node[color=cyan] at (1.5, 9.5){2}; 
\node[color=cyan] at (2.5, 9.5){3}; 
\node[color=cyan] at (3.5, 9.5){4}; 
\node[color=cyan] at (4.5, 9.5){5}; 

\node  [color=blue] at (.5, 8.5){2}; 
\node[color=cyan] at (1.5, 8.5){6}; 
\node[color=cyan] at (2.5, 8.5){7}; 
\node[color=cyan] at (3.5, 8.5){8}; 
\node[color=cyan] at (4.5, 8.5){9}; 

\node [color=blue]at (.5, 7.5){3}; 
\node[color=blue] at (1.5, 7.5){10}; 
\node[color=cyan] at (2.5, 7.5){11}; 
\node[color=cyan] at (3.5, 7.5){12}; 
\node[color=cyan] at (4.5, 7.5){15}; 

\node [color=blue]at (.5, 6.5){4}; 
\node[color=blue] at (1.5, 6.5){11}; 
\node[color=purple] at (2.5, 6.5){13}; 
\node[color=cyan] at (3.5, 6.5){14}; 
\node[color=cyan] at (4.5, 6.5){17}; 

\node [color=blue] at (.5, 5.5){5}; 
\node  [color=blue] at (1.5, 5.5){12}; 
\node[color=blue] at (2.5, 5.5){15}; 
\node[color=cyan] at (3.5, 5.5){16}; 
\node[color=cyan] at (4.5, 5.5){19}; 

\node [color=blue] at (.5, 4.5){6}; 
\node  [color=blue] at (1.5, 4.5){14}; 
\node[color=blue] at (2.5, 4.5){17}; 
\node[color=cyan] at (3.5, 4.5){18}; 
\node[color=purple] at (4.5, 4.5){22}; 

\node [color=blue] at (.5, 3.5){7}; 
\node  [color=blue] at (1.5, 3.5){16}; 
\node[color=blue] at (2.5, 3.5){19}; 
\node[color=cyan] at (3.5, 3.5){20}; 
\node[color=cyan] at (4.5, 3.5){24}; 

\node [color=blue]  at (.5, 2.5){8}; 
\node [color=blue] at (1.5, 2.5){18}; 
\node[color=purple] at (2.5, 2.5){21}; 
\node[color=cyan] at (3.5, 2.5){23}; 
\node[color=cyan] at (4.5, 2.5){25}; 

\node [color=blue]  at (.5, 1.5){9}; 
\node [color=blue] at (1.5, 1.5){20}; 
\node[color=blue] at (2.5, 1.5){24}; 
\node[color=purple] at (3.5, 1.5){26}; 
\node[color=cyan] at (4.5, 1.5){27}; 

\node[color=blue] at (.5, .5){10}; 
\node[color=blue] at (1.5, .5){23}; 
\node  [color=blue] at (2.5, .5){25}; 
\node[color=blue] at (3.5, .5){27}; 
\node[color=purple] at (4.5,.5){28}; 
\end{scope}

\end{tikzpicture}
\caption{ Fillings when $\alpha=5$.}
\label{fig:5col}
\end{figure}

\end{proof}


\section{Existence of components of codimension $-\rho$.}

We now use the results of the previous section to show the existence of components of the Brill-Noether loci of the expected dimension. 
We will use the techniques developed in \cite{LOTZ1}, \cite{LOTZ2}, \cite{Ramif} to show that at a generic point of this component, the Petri map is onto.

Assume that we have  a one-dimensional family of curves in which the generic curve is non-singular and the special curve is a chain of elliptic curves.
A line bundle in the family can be modified with line bundles that have support on the central fiber. This operation  changes the distribution of degree among the elliptic components.
The space of sections will give rise in each case to a space of sections on the elliptic chain.
When the degree is concentrated one at a time on each of the elliptic components, one obtains the data of a  limit linear series. 
One could instead consider all possible  distributions of degrees among the components.
The data of a linked linear series  (see \cite{O1}, \cite{O2}) is given by a line bundle on the chain 
and  spaces of sections of dimension $r+1$ for each such line bundles that are compatible with the maps between the line bundles obtained by changing degrees. 

In \cite{Ramif}, Definition 1.6, we defined a limit section of a limit linear series on a chain of elliptic curves as a collection of sections  $s_i$, one for each elliptic curve 
such that the orders of vanishing at the nodes satisfy $ord_{Q^i}s^i+ord_{P^{i+1}}s^{i+1}\ge d$.
We called the section  refined if the inequalities are equalities.
From a refined section, we can produce a section of a linked linear series for any choice of degrees (see section 2 of  \cite{LSVB}).
We showed in  \cite{Ramif}, Lemma  1.7 that for any limit section of the canonical series on a chain of elliptic and rational curves, 
there exists (at least one) elliptic component in which the order of vanishing of the section at the two nodes adds up to $2g-2$.
Moreover, given a collection of limit sections  such that  the orders of vanishing of the sections adding to $2g-2$ happen at distinct elliptic components, then these limit sections are linearly independent
(see  \cite{Ramif}, Proposition  1.8).
Note that Proposition  1.8 in  \cite{Ramif} is stated for a generic chain of elliptic curves but the proof is equally applicable when some of the nodes differ in torsion.

\begin{theorem}\label{ThExComp}
Fix  values of the genus $g, r, d$ such that  $r+1\le g-d+r$ and  $ -\rho(g,r,d)+r+1\le g,  \rho(g,r,d)\le 0$. 
Then there exists at least one component of ${\mathcal M}^r_{g,d}$ of codimension $-\rho$ in ${\mathcal M}_g$.
 The Petri map at the generic point of this component is onto.
\end{theorem}
\begin{proof} Write $\alpha=r+1, \beta=g-d+r$. Then $\rho=g-\alpha\beta$.
Moreover, the conditions   $r+1\le g-d+r$ and $ -\rho(g,r,d)+r+1\le g,  \rho(g,r,d)\le 0$ can be written as $\alpha\le \beta$, and $ \frac{\alpha(\beta+1)}2 \le g\le \alpha\beta$.
Then,  Lemma \ref{lemfilex}  guarantees the existence of a filling of the $\alpha\times \beta$-rectangle with the indices 
$1, 2,\dots, \alpha \beta+\rho$. 
As there are $ \alpha \beta$ spots in the rectangle, $-\rho$ of the  indices in the filling are repeated.
From Proposition \ref{prop:compesp}, such a filling corresponds to limit linear series on chains of elliptic curves  such that 
on $-\rho$ of the curves in the chain (corresponding to the repeated indices) the difference of the nodes is a torsion point on the elliptic curve of order determined by the grid distance between the two appearances of the index.
Therefore, the chains of elliptic curves on which such a linear series exists  moves on a space of codimension precisely $-\rho$ in the space of all chains of elliptic curves.
Note that $-\rho$ is the expected codimension of the loci of curves that posses a linear series with the given $d,r$.
From the determinantal nature of the space of linear series,  the sections extend to nearby non-singular curves. 
Therefore, there exists  at least one component of  ${\mathcal M}^r_{g,d}$ of codimension $-\rho$ in ${\mathcal M}_g$.

We now show that the Petri map is onto, or equivalently, that it has maximal rank.

From each column of the filled rectangle we can build a limit linear section $s_i, i=1,\dots, r+1$ such that 
\[ ord_{P_1}s_i= i-1, ord_{Q_g}s_i=r-i+1, \ ord_{Q_t}s_i+ord_{P_{t+1}}s_i=d, \   \ ord_{P_n}s_i+ord_{Q_n}s_i=\begin{cases} d& \text{ if } n\text{  is in column }\ \ \ \ i\\ d-1& \text{ if } n\text{ is  not in column }i\end{cases} \]
Interchanging rows and columns for an admissible filled rectangle, we obtain another admissible filled rectangle.
We will call this rectangle the Serre dual rectangle.
The linear series of degree $d$ and dimension $r$ then has a Serre dual linear series  of degree $\bar d=2g-2-d$ and dimension $\bar r=g-d+r-1$.
As grid distances are preserved when interchanging rows and columns, the Serre dual series works on the same special chain as the original one.
 We denote by $\bar s_1,\dots, \bar s_{\bar r+1}$  the sections corresponding to the columns of the Serre dual linear series.
For every index $i=1,\dots, g$, we choose  one spot (of the potentially two spots) where it appears in the rectangle.
If the chosen spot for the index $i$ is  $(a,b)$, we define the section of the canonical $t_i=s_a\bar  s_b$.
As we are assuming that the  index $i$ is in spot $(a,b)$ of the rectangle, the sections $s_a, \bar s_b$ have vanishing at the nodes on $C_i$ adding to the maximum:
\[    \ ord_{P_i}s_a+ord_{Q_i}s_a=d, \ ord_{P_i}\bar s_b+ord_{Q_i}\bar s_b=\bar d =2g-2-d.\]
Therefore, as  $t_i= s_a\bar s_b$,
\[ \ ord_{P_i}(t_i)+ord_{Q_i}(t_i)= \ ord_{P_i}s_a+ord_{P_i}\bar s_b+ord_{Q_i}s_a+ord_{Q_i}\bar s_b=2g-2\]
From our  choice of $s_a, \bar s_b$ and therefore of $t_i$, every index $i=1,\dots, g$ is being assigned a unique section $t_i$. 
From Proposition1.8 in  \cite{Ramif},  the $t_i$ are linearly independent.
As there are  $g$ different sections $t_i$,one for every index $1\le i\le g$, there are $g$ independent sections. 
Because $g$ is the the dimension of $H^0(C,K)$,  the Petri map has maximal rank.
\end{proof}

\begin{remark} Recall that the orthogonal to the image of the Petri map is identified to the tangent space to $W^r_d$.
The Petri map being onto confirms that $W^r_d$ is $0$-dimensional on curves near the elliptic chain.

Moreover, we can explicitly identify the elements in the Kernel of the Petri map at our special curves.
If we use the notations of   Theorem \ref{ThExComp}, for each  index $i$ appearing twice in the filling of the rectangle on spots $(a_1,b_1), (a_2,b_2)$, 
there exists a constant $c_{a,b}$ such that  $S_i=s_{a_1} \otimes \bar s_{b_1}-c_{a,b}s_{a_2} \otimes \bar s_{b_2}$ is in the kernel of the Petri map.
Moreover, these $S_i$ are linearly independent, as the $s_{a_1}\otimes \bar s_{b_1}$ are linearly independent.
As there is the right number of $S_i$, they span the kernel of the Petri map.

 One can find the tangent space to the locus where the Petri map fails to be injective inside the tangent space to   ${\mathcal M}_g$ 
as the orthogonal to the ramification divisors of the pencil spanned by $s_{a_1}, s_{a_2}$ (see \cite{AC}, \cite{semican} )
each pair $(a_1,b_1)$ imposes a different condition as the ramification divisors live on different elliptic components.
This may perhaps be helpful when trying to distinguish between different Brill-Noether loci.
\end{remark}


\section{Components of the same codimension are different}

We want to look  at possible equalities of loci ${\mathcal M}^r_{g,d}$  when these loci  are of  the same expected dimension.
It was proved in \cite{CKK}, \cite{CK}, that when the codimension is 1 or 2, then the loci are distinct (see also\cite{AHL}).
We generalize this result here to higher codimension. 
Note that by Serre duality, ${\mathcal M}^r_{g,d}={\mathcal M}^{r+g-d-1}_{g,2g-2-d}$, therefore, assuming that $r+1\le g-d+r$ is not  restrictive.

\begin{theorem}\label{distBN}
Fix  values of the genus $g, r_i, d_i$ such that  $\rho(g,r_i, d_i)< 0$.
Write $e_i=-\rho(g,r_i, d_i)$.
If $d_1,r_1, d_2, r_2$ are such that $e_1=e_2$ and  for each $i$, either $r_i+1\le g-d_i+r_i-2$ and $e_i\le \frac{r_i^2+3r_i-2}2$ or $r_i+1\le g-d_i+r_i$ and $e_i\le \frac{(r_i^2+r_i)}2$
then either $(r_1, d_1)=(r_2, d_2)$  or ${\mathcal M}^{r_1}_{g,d_1}\not= {\mathcal M}^{r_2}_{g,d_2}$.

\end{theorem}
\begin{proof}
The conditions   $\rho(g,r_1,d_1)=\rho(g,r_2,d_2)$, is equivalent to  $( r_1+1) (g-d_1+r_1)=( r_2+1) (g-d_2+r_2)$.
If also $( r_1+1)+ (g-d_1+r_1)=( r_2+1)+ (g-d_2+r_2)$,  then, as the sum and product of two numbers determine the two numbers up to the order, 
 either $r_1+1=r_2+1, g-d_1+r_1=g-d_2+r_2$ or $r_1+1=g-d_2+r_2, g-d_1+r_1=r_2+1$.
The first option means that  $d_1=d_2, r_1=r_2$. The second option translates into  $d_1=2g-2-d_2, r_1+1=g-d_2+r_2$.
So, the two Brill-Noether loci correspond to the same values of degree and dimension or to the  Serre dual values.

Assume then that $( r_1+1)+ (g-d_1+r_1)\not= ( r_2+1)+ (g-d_2+r_2)$.
To fix ideas, we can assume that  $( r_1+1)+ (g-d_1+r_1)> ( r_2+1)+ (g-d_2+r_2)$.
 With the notation of Lemma \ref{lemseptor}, as $e_1=e_2=e$, also $k_1=k_2=k, j_1=j_2=j$.
 From Lemma \ref{lemexfilcorn} and Lemma \ref{lemseptor},  there is an admissible fillings of the  $( r_1+1)\times (g-d_1+r_1)$-rectangle with the numbers $1,2,\dots, ( r_1+1)(g-d_1+r_1)-e$ for which 
 the sums of grid distances $|b^{i_t}_{1,2}-b^{i_t}_{1,1}1|+|a^{i_t}_{1,2}-a^{i_t}_{1,1}|$ between the spots where the indices $i_t, t=1,\dots, e$ appear   satisfies 
 \[ \sum_{t=1}^{e}(|b^{i_t}_{1,2}-b^{i_t}_{1,1}|+|a^{i_t}_{1,2}-a^{i_t}_{1,1}|)=  e(g+2r_1-d_1-2)-2[2\times 1+3\times 2+\dots+k(k-1)+jk]=A_1.       \]
In a filling of the  $( r_2+1)\times (g-d_2+r_2)$-rectangle with the numbers $1,2,\dots, ( r_2+1)(g-d_2+r_2)-e $,  the sums of grid distances of repeated digits satisfy 
 \begin{equation}\label{eqin} \sum_{t=1}^{e_2}(|b^{i_t}_{2,2}-b^{i_t}_{2,1}|+|a^{i_t}_{2,2}-a^{i_t}_{2,1}|)\le e(g+2r_2-d_2-2)-2[2\times 1+3\times 2+\dots+k(k-1)+jk]<   A_1       .       \end{equation}
Consider a chain of elliptic curves for which there exists a limit linear series of degree $d_1$ and rank $r_1$ with maximal sum of grid distances $A_1$ among repeated indices.
From inequality \ref{eqin}, such a curve does not support a limit linear series of rank $r_2$ and degree $d_2$.
Therefore,  ${\mathcal M}^{r_1}_{g,d_1}\nsubseteq {\mathcal M}^{r_2}_{g,d_2}$.\end{proof}

The methods we use in the previous proof can also be used in the case of different expected codimensions to show that  certain inclusions of Brill-Noether loci are impossible.
Some results in this direction were proved in \cite{CK} Corollary 3.6 and in very recent work (see  \cite{AHL}).
While this tool is unlikely to prove Auel-Haburcak conjecture, it provides some evidence that inclusion of potentially maximal Brill-Noether loci are rare.
We  give only one example below, many more could be investigated in a similar way.

\begin{example}\label{exinclloci}
Assume that   ${\mathcal M}^{r_1}_{g,d_1}\subseteq {\mathcal M}^{r_1+1}_{g,d_2}$ with the two loci being of codimension 2 and 1 respectively in ${\mathcal M}_g$.
 
 The conditions on codimension mean that $ g-(r_1+1)(g-d_1+r_1)=-2, \ g-(r_2+1)(g-d_2+r_2)=-1$.
 We can construct a smoothable $g^{r_1}_{d_1}$ on a chain of elliptic curves in which two of the curves have the differences of their nodes to be torsion of order $r_1+1+g-d_1+r_1-3$.
 If such a chain is in the boundary of ${\mathcal M}^{r_2}_{g,d_2}$, we need $r_1+1+g-d_1+r_1-3\le r_2+1+g-d_2+r_2-2$.
 Writing $\alpha_i=r_i+1, \beta_i=g-d_i+r_i$, we have the equations
 \[ \alpha_1\beta_1=\alpha_2\beta_2+1, \ \ \alpha_2=\alpha_1+1 ,\ \  \alpha_1+\beta_1\le \alpha_2+\beta_2+1 \]
 From the first two equations
  \[ \alpha_1\beta_1-(\alpha_1+1)\beta_2=1 .\]
 The general solution $\beta_1, \beta_2$ of  the above diophantine equation is
  \[ \beta_1= 2\alpha_1+1+t(\alpha_1+1), \beta _2=2\alpha_1-1+t\alpha_1\]
   for  integral values of $t$. As $\beta_i>0$,  we need $t\ge -1$.
   The inequality above can then be written as 
   \[ \alpha_1+ 2\alpha_1+1+t(\alpha_1+1)\le \alpha _1+1+2\alpha_1-1+t\alpha_1+1\Leftrightarrow  t\le 0 \]
  So, the only possible solutions are $t=0, t=-1$.
  
   If $t=0$, we have a potential inclusion 
   \[ {\mathcal M}^{\alpha -1}_{2\alpha^2+\alpha-2,2\alpha^2-4}\subseteq {\mathcal M}^{\alpha}_{2\alpha^2+\alpha-2,2\alpha^2-1} \]
   For $\alpha=2$, this give the known inclusion $ {\mathcal M}^{1}_{8,4}\subseteq {\mathcal M}^{2}_{8,7} $ while for $\alpha=3, 4\dots $, it would lead to potential inclusions 
   $ {\mathcal M}^{2}_{19,14}\subseteq {\mathcal M}^{3}_{19,17} ,\  {\mathcal M}^{3}_{34,28}\subseteq {\mathcal M}^{4}_{34,31} ,\dots$.
   The potential inclusion for the case of $g=19$ was proved not to materialize in \cite{AH} Proposition 6.17.
 
   If $t=-1$, we have a potential inclusion 
   \[ {\mathcal M}^{\alpha -1}_{\alpha^2-2,\alpha^2-3}\subseteq {\mathcal M}^{\alpha}_{\alpha^2-2,\alpha^2-1} ={\mathcal M}^{\alpha-2}_{\alpha^2-2,\alpha^2-5} \]
 where the last equality is obtained from Serre duality.
   We need the degree $\alpha^2-3$ to be larger than 1, so we take $\alpha$ to be at least three.
   For $\alpha=3$, we obtain the known inclusion  $ {\mathcal M}^{2}_{7,6}\subseteq {\mathcal M}^{3}_{7,8}={\mathcal M}^{1}_{7,4}$. 
  For $\alpha= 4, 5\dots $, it would lead to potential inclusions 
   $ {\mathcal M}^{3}_{14,13}\subseteq {\mathcal M}^{2}_{14,9} ,\  {\mathcal M}^{4}_{23,22}\subseteq {\mathcal M}^{3}_{23,20} ,\dots$.
   Work of Lelli-Chiesa \cite{L}, Auel and Haburcak \cite{AH}, and Auel, Haburcak and Larson \cite{AH} show that some of the potential inclusions do not exist, many others remain open. 
    Still, this method show that potential inclusions of Brill-Noether loci are few and far between.
 \end{example}

\section{Conflict of interests and data}
 This manuscript does not make use of any data.
 
 The author has no conflicts of interests.


\end{document}